\documentclass[a4paper,10pt,reqno]{amsart}
\usepackage[utf8]{inputenc}
\usepackage{amsmath,amsfonts,amssymb,mathrsfs,latexsym,bbm,enumitem,xcolor}
\usepackage[all,cmtip]{xy}
\newtheorem{definition}{Definition}
\newtheorem{theorem}[definition]{Theorem}
\newtheorem{lemma}[definition]{Lemma}
\newtheorem{corollary}[definition]{Corollary}
\newtheorem{proposition}[definition]{Proposition}
\newtheorem{remark}[definition]{Remark}

\renewcommand{\H}{{\mathrm{H}}}
\begin{document}
\title{On Tate Milnor-Witt Motives}
\author{Jean Fasel}
\address{Jean Fasel\\Institut Fourier-UMR 5582 \\ Universit\'e Grenoble-Alpes, CNRS \\ CS 40700 \\ 38058 Grenoble Cedex 9 \\ France}
\email{jean.fasel@univ-grenoble-alpes.fr}
\author{Nanjun Yang\(^\dag\)}
\thanks{\(^\dag\)Partially supported by the National Natural Science Foundation of China (Grant No. 12201336)}
\address{Nanjun Yang\\Beijing Institute of Mathematical Sciences and Applications\\Huairou District\\Beijing China}
\email{ynj.t.g@126.com}

\begin{abstract}    
Smooth projective \(\mathbb{G}_m\)-varieties with isolated rational fixed points admit Tate Milnor-Witt motives. Over Euclidean fields, we give a splitting formula of such motives, which reduces the computation of their Chow-Witt groups to that of their Chow groups and cohomologies of Witt sheaf.
\end{abstract}

\maketitle

\section{Introduction}\label{Introduction}
Throughout the text, we work over a base field \(k\) which is infinite, perfect and has characteristic different from \( 2\).
In that context, recall that one can define the category of Milnor-Witt motives \(\widetilde{\mathrm{DM}}(k,\mathbb{Z})\) (aka. MW-motives, \cite{BCDFO}) which is a refinement of Voevodsky's category of motives \(\mathrm{DM}(k,\mathbb{Z})\) in the sense that the adjunction 
\[
\mathrm{SH}(k)\leftrightarrows \mathrm{DM}(k,\mathbb{Z})
\]
factorizes into adjunctions
\[
\mathrm{SH}(k)\leftrightarrows\widetilde{\mathrm{DM}}(k,\mathbb{Z}) \leftrightarrows \mathrm{DM}(k,\mathbb{Z}).
\]
We refer the reader to the introduction of \cite{BCDFO} for more information, and only state one of the main properties of the middle category. Its tensor unit, written $\mathbb{Z}$ for simplicity (but sometimes $\mathrm{H}\tilde{\mathbb{Z}}$ or $\mathbf{H}_{\mathrm{MW}}$ elsewhere in the literature), represents the so-called \emph{Milnor-Witt motivic cohomology} in the sense that
\[
[\tilde{M}(X),\mathbb{Z}(j)[i]]_{\widetilde{\mathrm{DM}}(k,\mathbb{Z})}\simeq \mathrm{H}^{i,j}_{\mathrm{MW}}(X,\mathbb{Z})
\]
where $\tilde{M}(X)$ is the MW-motive associated to any smooth scheme $X$. If $X$ is defined over the field of real numbers, there exists complex and real realization functors that yield homomorphisms 
\[
t_{\mathbb{C}}:\mathrm{H}^{i,j}_{\mathrm{MW}}(X,\mathbb{Z})\to \mathrm{H}^{i,j}_{\mathrm{M}}(X_{\mathbb{C}},\mathbb{Z})\to \mathrm{H}_{\mathrm{sing}}^{2i}(X(\mathbb{C}),\mathbb{Z})
\]
where the right-hand group is the singular cohomology group associated to the analytic space $X(\mathbb{C})$ and the middle group is the ordinary motivic cohomology group associated to the scheme $X_{\mathbb{C}}:=X\times_{\mathrm{Spec}(\mathbb{R})}\mathrm{Spec}(\mathbb{C})$, and
\[
t_{\mathbb{R}}:\mathrm{H}^{i,j}_{\mathrm{MW}}(X,\mathbb{Z})\to  \mathrm{H}_{\mathrm{sing}}^{i-j}(X(\mathbb{R}),\mathbb{Z})
\] 
where now the right-hand side denotes the singular cohomology group associated to the real manifold $X(\mathbb{R})$. These homomorphisms somewhat explain why even quite simple cellular varieties usually have a non free Milnor-Witt cohomology, indeed it has to take into account that this is not the case over the real numbers, where cells can be glued into a non-trivial way.  

In \cite{Y1} and \cite{Y2}, it was however shown that the MW-motives of Grassmannians and complete flags are direct sums of \(\mathbb{Z}(i)[2i]\) and \(\mathbb{Z}/\eta(i)[2i]\), where the latter can be though of a motivic analogue of coefficients modulo $2$.  As a result, as we will see below (Proposition \ref{blowup}), computing the MW-motives of projective bundles and blowups is tantamount to computing extensions of Thom spaces of the base scheme. If we keep doing such operations from a point, it amounts to study extensions of \(\mathbb{Z}(i)[2i]\).

Another important situation is when the projective scheme \(X\in \mathrm{Sm}/k\) admits a \(\mathbb{G}_m\)-action with isolated rational fixed points. Then, the well-known Bia\l ynicki-Birula decomposition says that there is a filtration
\[\emptyset=X_{-1}\subseteq X_0\subseteq\cdots\subseteq X_n=X\]
of closed subsets such that \(X_i\setminus X_{i-1}\) are affine spaces. Applying the Gysin triangle to each successive embedding \(X_i\setminus X_{i-1}\subseteq X\setminus X_{i-1}\) shows that all Thom spaces of (the MW-motive of) \(X\) come from extensions of \(\mathbb{Z}(i)[2i]\). We call such kind of MW-motives \textit{Tate} (see Definition \ref{tatedef} for more details).

What do Tate MW-motives look like? Suppose that we have a distinguished triangle
\[A\longrightarrow B\longrightarrow C\xrightarrow{\partial} A[1]\]
where \(A,B,C\) are Tate. In the category of Voevodsky's motives \(\mathrm{DM}(k,\mathbb{Z})\), the triangle splits since
\[[\mathbb{Z}(i)[2i],\mathbb{Z}(j)[2j+1]]_{DM}=0\]
for \(i,j\in\mathbb{N}\) and thus any Tate motive in \(\mathrm{DM}(k,\mathbb{Z})\) is a direct sum of \(\mathbb{Z}(i)[2i]\), i.e. split. In the category of MW-motives  \(\widetilde{\mathrm{DM}}(k,\mathbb{Z})\), we however have
\[[\mathbb{Z}(i)[2i],\mathbb{Z}(j)[2j+1]]_{\widetilde{\mathrm{DM}}(k,\mathbb{Z})}=\begin{cases}0&\text{ if $i\neq j+1$},\\\mathbf{W}(k)\cdot\eta&\text{ if $i= j+1$},\end{cases}\]
where the \(\mathbf{W}(k)\) is the Witt group of the field \(k\) and \(\eta:\mathbb{Z}(1)[1]\longrightarrow\mathbb{Z}\) is the Hopf element. If the former group is not a field, the question becomes nontrivial since one is led to consider all ideals in \(\mathbf{W}(k)\). In case \(k\) is Euclidean (an assumption we will keep for the rest of the paper), namely an ordered field in which all positive elements admit a square root, we have \(\mathbf{W}(k)=\mathbb{Z}\). Then it is natural to consider the blocks \(\mathbb{Z}/l\eta(i)[2i],l\in\mathbb{Z}\), namely the mapping cone of \(l\eta\). Writing \(l=2^ts\) for \(s\) odd, we obtain a decomposition
\[\mathbb{Z}/l\eta=\mathbb{Z}/2^t\eta\oplus\mathbb{Z}/\mathbf{s}\]
where \(\mathbf{s}=\sum_{i=0}^{s-1}\epsilon^i,\epsilon=-<-1>\) which is enough to describe all the Tate motives. Our main result (Theorem \ref{tate}) then reads as follows:
\begin{theorem}
Every Tate MW-motive \(A\) (over Euclidean fields) is a direct sum of \(\mathbb{Z}(i)[2i]\), \(\mathbb{Z}/2^t\eta(i)[2i]\) and \(\mathbb{Z}/\mathbf{s}[i]\) where \(s\) is odd.
\end{theorem}
The theorem shows that one can reduce the computation of MW-motivic cohomologies of such motives to that of the cohomologies of Witt sheaf, since in practice the Chow groups are usually well-known (e.g. flag varieties of type A (\cite{Y2}) and D (\cite{HMX}), \(\overline{M}_{0,n}\) (\cite{T})).

Before explaining the rough lines of the proof, let us say a word on the \emph{a priori} restrictive assumption that $k$ is Euclidean. Most of the objects we consider in this article are defined over $\mathbb{Z}$, and the still elusive construction of $\widetilde{\mathrm{DM}}(\mathbb{Z},\mathbb{Z})$ should satisfy all the main properties of its known version over a field. In particular, the result about extensions of split objects above should prove true, i.e. 
\[[\mathbb{Z}(i)[2i],\mathbb{Z}(j)[2j+1]]_{\widetilde{\mathrm{DM}}(\mathbb{Z},\mathbb{Z})}=\begin{cases}0&\text{ if $i\neq j+1$},\\ \mathbf{W}(\mathbb{Z})\cdot\eta&\text{ if $i= j+1$},\end{cases}\]
and one could infer from $ \mathbf{W}(\mathbb{Z})=\mathbb{Z}$ that the results of this paper should apply to \emph{all} MW-motives of objects defined over $\mathbb{Z}$ such as partial flag varieties, etc.

Our main tool for proving the above theorem is a careful analysis of the Bockstein spectral sequence arising from the exact couple
\begin{equation}\label{eqn:Bockstein}
	\xymatrix
	{
		\mathrm{H}^{*,*}_{\mathrm{MW}}(-,\mathbb{Z})\ar[rr]^{\eta}	&																				&\mathrm{H}^{*,*}_{\mathrm{MW}}(-,\mathbb{Z})\ar[ld]\\
																						&E_1^{*,*}(-)\ar[lu]	&
	},
\end{equation}
where 
\[E_1^{p,q}(X)=[X,\mathbb{Z}/\eta(q)[p]]_{\widetilde{\mathrm{DM}}(k,\mathbb{Z})}.\] 
This spectral sequence is the motivic version of the Pardon spectral sequence discussed e.g. in \cite{AF4}, with the advantage that all differentials are known explicitly rather than only the diagonal part. We give in Proposition \ref{1} and Corollary \ref{identification} a general method computing MW-motivic cohomologies in terms of motivic cohomologies and cohomologies of Witt sheaf, given a uniform upper bound of the \(\eta\)-torsions. The method is particularly effective when there is no genuine \(\eta^2\)-torsions, which is expected to be true in the case of proper smooth curves over local fields and \(\mathbb{R}\) (on going work). For a recent conjecture on the upper bound of \(2\)-torsions, see \cite{MW}.

The main theorem shows that in the context above, the spectral sequence degenerates at the \(E_r\)-page (\(r\geq 2\)) if and only if all \(l\) are of the form \(2^{r'-2}s,2\leq r'\leq r\) where \(s\) is odd. From the second page, there is a K\"unneth property for Tate MW-motives (see Proposition \ref{bockstein compare}), which extends the classical result in \cite{Bro}.
\begin{proposition}
Let \(\rho=[-1]\in\mathbf{K}_1^{\mathrm{MW}}(k)\) and consider \(E_i(A)^a:=\oplus E^{*+a,*}_i(A)\) as a \(\mathbb{Z}/2[\rho]\)-module chain complex. We have
\[E_2(A)\otimes_{\mathbb{Z}/2[\rho]}E_2(B)=E_2(A\otimes B)\]
as complexes and
\[E_i(A)\otimes^L_{\mathbb{Z}/2[\rho]}E_i(B)=E_i(A\otimes B)\]
in \(D(\mathbb{Z}/2[\rho]\textrm{-mod})\) for every \(i\geq 2\).
\end{proposition}

Moreover, there is a Cartesian square (see Proposition \ref{bockstein compare})
\[
	\xymatrix
	{
		\mathrm{H}^{*,*}_{\mathrm{MW}}(A,\mathbb{Z})\ar[r]\ar[d]	&\mathrm{H}^{*,*}_{\mathrm{M}}(A,\mathbb{Z})\ar[d]\\
		\mathrm{H}^{*,*}_{\mathrm{W}}(A,\mathbb{Z})\ar[r]				&\mathrm{H}^{*,*}_{\mathrm{M}}(A,\mathbb{Z}/2),
	}
\]
which extends the result in \cite{HW}. Here, we write $\mathrm{H}^{*,*}_{\mathrm{MW}}, \mathrm{H}^{*,*}_{\mathrm{M}}$ and $\mathrm{H}^{*,*}_{\mathrm{W}}$ for respectively MW-motivic, motivic and \(\mathbf{I}^j\)-motivic cohomologies.  

To find a scheme with Tate MW-motive and a nontrivial Bockstein in arbitrary high pages, we consider \(\mathbb{P}(E)\) where \(E\) is a vector bundle on \(\mathbb{HP}^1\) (see \cite{PW} for definition). If we denote by \(\mathcal{V}_n(X)\) the set of isomorphism classes of vector bundle of rank \(n\) on \(X\in \mathrm{Sm}/k\), our result is the following (see Corollary \ref{vb}):
\begin{theorem}
The Euler class induces a bijection
\[\mathcal{V}_2(\mathbb{HP}^1)\cong\mathbf{GW}(k)/k^{\times}\]
and the second Chern class induces a bijection
\[\mathcal{V}_n(\mathbb{HP}^1)\cong\mathbb{Z}, n\geq 3.\]
\end{theorem}
So any vector bundle \(E\) of rank two whose Euler class \(e(E)\) has image \(2^n\in\mathbf{W}(k)\) gives a nontrivial Bockstein in the \((n+1)\)st-page of \(\mathbb{P}(E)\).

For convenience, we give a list of frequently used notations in this paper:
\[\begin{array}{c|c}
\mathrm{Sm}/k	&\textrm{The category of smooth and separated schemes over \(k\)}\\
\mathbb{Z}(X)		&\textrm{Motive of \(X\)}\\
C(f)		&\textrm{Mapping cone of \(f\)}\\
\pi		&\textrm{The reduction modulo \(2\) map}\\
\eta		&\textrm{The Hopf element}\\
\epsilon&-<-1>\\
\tau		&\textrm{The motivic Bott element}\\
\iota		&\textrm{The quotient map of }\tau[1]\\
\rho		&[-1]\in\mathbf{K}_1^{MW}(k)\\
A/l\eta	&A\otimes C(l\eta)
\end{array}\]

\section{The MW-Bockstein Spectral Sequence and its Degeneracy}

\subsection{Reminders on exact couples}
By an exact couple \((D,E,i,j,k)\) we mean a diagram of abelian groups
\[
	\xymatrix
	{
		D\ar[rr]^{i}	&	&D\ar[ld]^j\\
							&E\ar[lu]^{k}	&
	}
\]
such that there is an exact sequence
\[\cdots\xrightarrow{k}D\xrightarrow{i}D\xrightarrow{j}E\xrightarrow{k}D\xrightarrow{i}\cdots.\]
We refer to \cite[\S 5.9]{Wei94} for more information about these notions.
For \(n\geq 1\), we set \(D_n:=\mathrm{Im}(i^n)\), \(Z_n:=k^{-1}(\mathrm{Im}(i^n))\), \(B_n:=j(\ker(i^n))\) and \(E_{n+1}:=Z_n/B_n\) as usual (\(n\geq 1\)). Further, consider \(D[i^{-1}]:=\varinjlim(D\xrightarrow{i}D\xrightarrow{i}\cdots)\), endowed with its canonical homomorphism \(p:D\to D[i^{-1}]\) and set \(\overline{D}:=\mathrm{Im}(p)\). We note that $\ker(p)=\ker(i^{\infty}):=\cup_{n\in\mathbb{N}}\ker(i^n)$ and then that there is an exact sequence
\[
0\to \ker(i^\infty)\to D\xrightarrow{p} \overline D\to 0.
\]
Combining with the exact sequence
\[
0\to D_1\to D\xrightarrow{j}\ker(k)\to 0
\] 
and chasing in the relevant diagrams, we obtain an exact sequence
\[
0\to D_1\cap\ker(i^\infty)\to D\xrightarrow{(j,p)} \ker(k)\oplus \overline D\xrightarrow{ (\pi, -\overline j)} \ker(k)/j(\ker(i^\infty))\to 0
\]
where $\pi$ is the projection and $\overline j$ is induced by $j$. On the other hand, it is clear that the inclusion $\ker(k)\subset \cap_{n\in\mathbb{N}}Z_n$ induces an isomorphism
\[
c:\ker(k)/j(\ker(i^\infty))\to E_{\infty}.
\]
and we finally obtain an exact sequence
\begin{equation}\label{eqn:exactcouple}
0\to D_1\cap\ker(i^\infty)\to D\xrightarrow{(j,p)} \ker(k)\oplus \overline D\xrightarrow{ (c\circ\pi, -c\circ\overline j)} E_{\infty}\to 0.
\end{equation}

\begin{definition}
Let \((D,E,i,j,k)\) be an exact couple. We say that it has \emph{up to \(i^r\)-torsion} if \(ker(i^{r+1})=ker(i^r)\). If $r=1$, we simply say that the exact couple \((D,E,i,j,k)\) has up to \(i\)-torsion, in place of \(i^1\)-torsion.
\end{definition}

\begin{proposition}\label{1}
The following assertions are verified:
\begin{enumerate}
\item If \((D,E,i,j,k)\) has up to \(i^r\)-torsion for some \(r\geq 1\), the associated spectral sequence degenerates at the (\(r+1\))st-page, i.e. there is a canonical isomorphism $\phi_r:E_{r+1}\to E_{\infty}$. Further, the exact sequence \eqref{eqn:exactcouple} reduces to an exact sequence
\[
0\longrightarrow D_1\cap \ker(i^{\infty})\longrightarrow D\xrightarrow{(j,p)} \ker(k)\oplus\overline{D}\to E_{r+1}\longrightarrow 0.
\]
\item If \((D,E,i,j,k)\) has up to \(i\)-torsion, the exact sequence \eqref{eqn:exactcouple} induces a Cartesian square
\[
\xymatrix{D\ar[r]^-j\ar[d]_-p & \ker(k)\ar[d] \\
\overline D\ar[r] & E_2.}
\]
\item If \(\ker(j,p)=D_1\cap \ker(i^{r-1})\), then \((D,E,i,j,k)\) has up to \(i^r\)-torsion.
\end{enumerate}
\end{proposition}

\begin{proof}
Let us prove the first assertion. If \(x\in E_{r'+1},r'\geq r\), we have \(k(x)=i^{r'}(y)\), and then \(i^{r'+1}(y)=0\). So \(k(x)=0\) by the condition, and the spectral sequence degenerates at the (\(r+1\))st-page. The assertion on the exact sequence is a direct consequence of the first one and we pass to (2). It suffices to prove that $D_1\cap\ker(i^\infty)=0$ in that case to conclude. Now, we have $\ker(i^\infty)=\ker(i)$  and we are reduced to check that $\mathrm{Im}(i)\cap\ker(i)=0$: If $x=i(y)$ and $i(x)=0$, we obtain $i^2(y)=0$ and thus $x=i(y)=0$. 

For the last statement, suppose that $x\in D$ is such that $i^{r+1}(x)=0$. Then \(i(x)\in D_1\cap \ker(i^{\infty})=D_1\cap \ker(i^{r-1})\) and it follows that \(i^r(x)=0\).
\end{proof}

We now consider higher versions of the homomorphism $j$ which will allow us to give a sufficient condition for an element $x\in D$ to be trivial. For $n\geq 1$, observe that mapping $i^n(x)$ to the class $\overline{j(x)}$ in $E_{n+1}$ induces a well-defined homomorphism
\[
j^{(n)}:D_n\to E_{n+1}.
\] 
and we set $j^{(0)}:=j$. For $n\geq 1$,
 a direct computation shows that the  sequence
\[
0\to D_{n+1}\cap \ker(i^{\infty})\to D_{n}\cap \ker(i^{\infty})\xrightarrow{j^{(n)}}E_{n+1}
\]
is exact.
If $x\in D_{n}\cap \ker(i^{\infty})$, we have $x=i^n(y)$ and $i^s(x)=0$ for some $s\in\mathbb{N}$. Consequently, $y\in \ker(i^{\infty})$ and $j^{(n)}(x)\in j(\ker(i^{\infty})):=B_{\infty}$. Conversely, any element in $z\in B_{\infty}$ is of the form $z=j(x)$ for some $x\in \ker(i^{\infty})$ and we have $z=j^{(n)}(i^n(x))\in \mathrm{Im}(j^{(n)})$. Consequently, we have an exact sequence
\begin{equation}\label{filtration}
0\to D_{n+1}\cap \ker(i^{\infty})\to D_{n}\cap \ker(i^{\infty})\xrightarrow{j^{(n)}}E_{n+1}\to Z_n/B_{\infty}\to 0.
\end{equation}

\begin{corollary}\label{identification}
Suppose that \((D,E,i,j,k)\) has up to \(i^r\)-torsion, and let $x\in D$. Then, $x=0$ if and only if $x\in \ker(j^{(n)})$ for any $0\leq n<r$.
\end{corollary}

\begin{proof}
As $\ker(i^{\infty})=\ker(i^{r})$ if \((D,E,i,j,k)\) has up to \(i^r\)-torsion, we obtain $D_{n}\cap \ker(i^{\infty})=0$ for $n\geq r$ and the claim follows easily.
\end{proof}

\subsection{The Bockstein spectral sequence for Witt motivic cohomology}

We now provide some spectra on which we want to use the formalism developed above. For this, recall that Bachmann defined in \cite{B} the spectra \(\mathrm{H}\widetilde{\mathbb{Z}},\H_{\mu}\mathbb{Z},\H_{\mathrm{W}}\mathbb{Z},\H_{\mu}\mathbb{Z}/2\) as the respective effective covers of the homotopy modules \(\mathbf{K}^{\mathrm{MW}}_*,\mathbf{K}^{\mathrm{M}}_*,\mathbf{K}^{\mathrm{W}}_*,\mathbf{K}^{\mathrm{M}}_*/2\). It turns out that the spectra \(\H\widetilde{\mathbb{Z}}\) and \(\H_{\mu}\mathbb{Z}\) represent respectively the cohomology theories \(\H^{*,*}_{\mathrm{MW}}\) and \(\H^{*,*}_{\mathrm{M}}\) respectively (see e.g. \cite[Proposition 4.3]{Y1}). 

We consider the exact couple 
\[
	\xymatrix
	{
		\oplus_{p,q}\H_{\mathrm{W}}\mathbb{Z}\wedge S^{p,q}\ar[rr]^{\eta}	&																				&\oplus_{p,q}\H_{\mathrm{W}}\mathbb{Z}\wedge S^{p,q}\ar[ld]\\
																						&\oplus_{p,q}\H_{\mathrm{W}}\mathbb{Z}/\eta\wedge S^{p,q}\ar[lu]	&
	},
\]
which induces a Bockstein spectral sequence $E({\mathrm{W}})^{*,*}$ with
\[E({\mathrm{W}})_1^{p,q}(F_{\bullet})=\H^{p,q}_{\mathrm{M}}(F_{\bullet},\mathbb{Z}/2)\oplus \H^{p+2,q}_{\mathrm{M}}(F_{\bullet},\mathbb{Z}/2)\]
for any simplicial sheaf \(F_{\bullet}\) by the computation in \cite{B}. For further use, let us compute the differential map in the first page, which is determined by the composite
\[
\H_{\mathrm{W}}\mathbb{Z}/\eta\longrightarrow \H_{\mathrm{W}}\mathbb{Z}\wedge\mathbb{P}^1\longrightarrow (\H_{\mathrm{W}}\mathbb{Z}/\eta)\wedge\mathbb{P}^1,
\]
namely a homomorphism
\[
D=\begin{pmatrix}a&b\\c&d\end{pmatrix}:\H_{\mu}\mathbb{Z}/2\oplus \H_{\mu}\mathbb{Z}/2[2]\longrightarrow \H_{\mu}\mathbb{Z}/2\wedge\mathbb{P}^1\oplus \H_{\mu}\mathbb{Z}/2\wedge\mathbb{P}^1[2].
\]
It is proved in \cite[Proposition 2.3]{Y2} that \(a=Sq^2\) and \(b=\tau\), where $\tau$ is the motivic Bott element (i.e. the class corresponding to $-1$ under the identification $\{\pm{1}\}\simeq\H_{\mathrm{M}}^{0,1}(k,\mathbb{Z}/2)$), and we now compute \(c\) and \(d\). 
\begin{lemma}
The differential $D$ in the Bockstein spectral sequence for Witt motivic cohomology is given by
\[
D=\begin{pmatrix}Sq^2 &\tau \\ Sq^3Sq^1 & Sq^2+\rho Sq^1\end{pmatrix},
\]
where $\rho=\{-1\}\in \H_{\mathrm{M}}^{1,1}(k,\mathbb{Z}/2)$.
\end{lemma}

\begin{proof}
We know from the generators of the Steenrod algebra that
\[c=uSq^3Sq^1, d=vSq^2+wSq^1\]
where \(u,v\in \H^{0,0}_{\mathrm{M}}(k,\mathbb{Z}/2)\) and \(w\in \H_{\mathrm{M}}^{1,1}(k,\mathbb{Z}/2)\). Since \(D^2=0\), we have
\[
D^2\begin{pmatrix}1\\0\end{pmatrix}=D\begin{pmatrix}a\\c\end{pmatrix}=\begin{pmatrix}Sq^2Sq^2+\tau uSq^3Sq^1\\uSq^3Sq^1Sq^2+dc\end{pmatrix}=0.
\]
Since \(Sq^2Sq^2=\tau Sq^3Sq^1\), we have \(\tau(u+1)=0\) hence \(u=1\). On the other hand, the equations
\[
dc=(vSq^2+wSq^1)Sq^3Sq^1=vSq^2Sq^3Sq^1=vSq^5Sq^1
\]
\[
Sq^3Sq^1Sq^2=Sq^3Sq^3=Sq^5Sq^1
\]
imply that \(v=1\). Next, we get from
\[
D^2\begin{pmatrix}0\\1\end{pmatrix}=D\begin{pmatrix}b\\d\end{pmatrix}=\begin{pmatrix}Sq^2\tau+\tau Sq^2+\tau wSq^1\\\cdots\end{pmatrix}=0
\]
and
\[
Sq^2\tau=\tau Sq^2+\tau\rho Sq^1
\]
that \(\tau(w+\rho)=0\), and hence that \(w=\rho\). Finally, \(c=Sq^3Sq^1\) and \(d=Sq^2+\rho Sq^1\).
\end{proof}

In \cite[Definition 2.11]{Y2}, we defined the so-called Bockstein cohomology
\[
E^{p,q}(F_{\bullet})=\frac{\ker(Sq^2) (\textrm{mod }\tau)}{\mathrm{Im}(Sq^2) (\textrm{mod }\tau)}
\]
associated to any pointed simplicial sheaf \(F_{\bullet}\). We note that there is an obvious homomorphism 
\[
f_{p,q}:E_2({\mathrm{W}})^{p,q}(F_{\bullet})\to E^{p,q}(F_{\bullet})
\]
defined by $\overline{(x,y)}\mapsto\overline{x}$ where \(x\in \H^{p,q}_{\mathrm{M}}(F_{\bullet},\mathbb{Z}/2)\) and \(y\in \H^{p+2,q}_{\mathrm{M}}(F_{\bullet},\mathbb{Z}/2)\).

\begin{proposition}\label{w}
If the motivic Bott element \(\tau\) is a non-zero-divisor in \(\H^{*,*}_{\mathrm{M}}(F_{\bullet},\mathbb{Z}/2)\), the morphism 
\[
f_{p,q}\colon E_2({\mathrm{W}})^{p,q}(F_{\bullet})\to E^{p,q}(F_{\bullet})
\]
is an isomorphism for any $p,q\in\mathbb{Z}$.
\end{proposition}

\begin{proof}
For \((x,y)\in \H^{p,q}_{\mathrm{M}}(F_{\bullet},\mathbb{Z}/2)\oplus \H^{p+2,q}_{\mathrm{M}}(F_{\bullet},\mathbb{Z}/2)\), we have
\[\tau Sq^3Sq^1(x)+\tau Sq^2(y)+\tau\rho Sq^1(y)=Sq^2(Sq^2(x)+\tau y).\]
If \(Sq^2(x)+\tau y=0\) we obtain 
\[\tau(Sq^3Sq^1(x)+Sq^2(y)+\rho Sq^1(y))=0\]
and thus $Sq^3Sq^1(x)+Sq^2(y)+\rho Sq^1(y)=0$
as well since $\tau$ is a non-zero-divisor. It follows that $f_{p,q}$ is onto.  On the other hand, if \(\overline{(x,y)}=0\), there are \(a\in \H^{p-2,q-1}_{\mathrm{M}}(F_{\bullet},\mathbb{Z}/2)\) and \(b\in \H^{p,q-1}_{\mathrm{M}}(F_{\bullet},\mathbb{Z}/2)\) such that
\[Sq^2(a)+\tau b=x\]
 and then
\[\tau Sq^3Sq^1(a)+\tau Sq^2(b)+\tau\rho Sq^1(b)=Sq^2(Sq^2(a)+\tau b)=Sq^2(x)=\tau y.\]
So
\[Sq^3Sq^1(a)+Sq^2(b)+\rho Sq^1(b)=y\]
as well and the map is injective.
\end{proof}

We note here that the condition of the proposition is satisfied if \(F_{\bullet}\) is a direct sum of motives of the form \(\mathbb{Z}/2(q)[p]\) in \(\mathrm{DM}(k,\mathbb{Z}/2)\) by \cite[Theorem 6.17]{V1}.

\subsection{The Bockstein spectral sequence for Milnor-Witt motivic cohomology}

We now pass to our main object of study, i.e. the spectrum $\mathrm{H}\widetilde{\mathbb{Z}}$ (for whom a more explicit model can be found in \cite[Chapter 6, Section 4]{BCDFO}), obtaining a Bockstein spectral sequence $E(\mathrm{MW})^{p,q}$. 

In the sequel, we will denote by \(\pi:\H^{*,*}_{\mathrm{M}}(-,\mathbb{Z})\longrightarrow \H^{*,*}_{\mathrm{M}}(-,\mathbb{Z}/2)\) the reduction modulo \(2\) map and by $\iota$ the quotient by $\tau$ map, i.e. the map sitting in the exact triangle
\[
\H_{\mu}\mathbb{Z}/2[1]\xrightarrow{\tau[1]}\H_{\mu}\mathbb{Z}/2(1)[1]\xrightarrow{\iota} C(\tau)\to \H_{\mu}\mathbb{Z}/2[2].
\]

\begin{proposition}\label{inj}
Suppose that \(\tau\) is a non-zero-divisor in \(\H^{*,*}_{\mathrm{M}}(F_{\bullet},\mathbb{Z}/2)\). If we define \(R^{p,q}(F_{\bullet})\) by the Cartesian square
\[
	\xymatrix@C=3em
	{
		R^{p,q}(F_{\bullet})\ar[r]\ar[d]													&\H^{p+2,q+1}_{\mathrm{M}}(F_{\bullet},\mathbb{Z})\ar[d]^-{\iota\circ\pi}\\
		\H^{p,q}_{\mathrm{M}}(F_{\bullet},\mathbb{Z})\ar[r]_-{\iota\circ Sq^2\circ\pi}	&\H^{p+2,q+1}_{\mathrm{M}}(F_{\bullet},\mathbb{Z}/2)/\tau,
	}
\]
we have a short exact sequence
\[0\longrightarrow\frac{\H^{*+1,*+1}_{\mathrm{M}}(F_{\bullet},\mathbb{Z}/2)}{\mathrm{Im}(Sq^2\circ\pi)+\mathrm{Im}(\pi)+\mathrm{Im}(\tau)}\longrightarrow E_1({\mathrm{MW}})^{*,*}(F_{\bullet})\longrightarrow R^{*,*}(F_{\bullet})\longrightarrow 0.\]
\end{proposition}
\begin{proof}
By the proof of \cite[Theorem 4.13]{Y1}, there is a distinguished triangle
\[\H\widetilde{\mathbb{Z}}/\eta\longrightarrow \H_{\mu}\mathbb{Z}/\eta\longrightarrow C(\tau[1])\longrightarrow \H\widetilde{\mathbb{Z}}/\eta[1],\]
which induces an exact sequence
\begin{equation}\small\label{e}\cdots\longrightarrow E_1^{p,q}(F_{\bullet})\xrightarrow{i} \H^{p+2,q+1}_{\mathrm{M}}(F_{\bullet},\mathbb{Z})\oplus \H^{p,q}_{\mathrm{M}}(F_{\bullet},\mathbb{Z})\xrightarrow{(\iota\circ\pi,\iota\circ Sq^2\circ\pi)}[F_{\bullet},C(\tau[1])(q)[p]]_{\mathrm{DM}(k,\mathbb{Z}/2)}\longrightarrow\cdots.\end{equation}
Since \(\tau\) is a non-zero-divisor, the last term is identified with \(\H^{p+2,q+1}_{\mathrm{M}}(F_{\bullet},\mathbb{Z}/2)/\tau\) and the statement follows.
\end{proof}

\begin{definition}
We say that a MW-motive has up to \(\eta^r\)-torsion if its Bockstein spectral sequence in \(\widetilde{\mathrm{DM}}(k,\mathbb{Z})\) does.
\end{definition}

\begin{definition}\label{tatedef}
The category of \emph{mixed Tate MW-motives} is defined to be the smallest thick (triangulated) subcategory of \(\widetilde{\mathrm{DM}}(k,\mathbb{Z})\) containing \(\mathbb{Z}(i)[2i]\) for all \(i\in\mathbb{N}\). The category of \emph{Tate MW-motives} is defined to be the smallest full subcategory of \(\widetilde{\mathrm{DM}}(k,\mathbb{Z})\) containing \(\mathbb{Z}(i)[2i]\) for any \(i\in\mathbb{N}\) and closed under both extensions and direct summands.
\end{definition}

Recall from \cite[3.3.6.a]{BCDFO} that there is a functor 
\[
\gamma^*:\widetilde{\mathrm{DM}}(k,\mathbb{Z})\longrightarrow \mathrm{DM}(k,\mathbb{Z})
\]
mapping the MW-motive associated to a smooth scheme $X$ to its ordinary motive. 

\begin{proposition}\label{bockstein}
If \(\gamma^*F_{\bullet}\) is a direct sum of motives of the form \(\mathbb{Z}(q)[p]\) for $p,q\in\mathbb{N}$, we have
\[\H^{*,*}_{\mathrm{M}}(F_{\bullet},\mathbb{Z}/2)=\mathrm{Im}(\pi)+\mathrm{Im}(\tau)\]
and consequently \(E(\mathrm{MW})_1^{*,*}(F_{\bullet})=R^{*,*}(F_{\bullet})\). Moreover, for every \(i\geq 2\), we have
\[\begin{array}{cc}E(\mathrm{MW})_2^{*,*}(F_{\bullet})=E^{*,*}(F_{\bullet});&E(\mathrm{MW})_i^{*,*}(F_{\bullet})=E(\mathrm{W})_i^{*,*}(F_{\bullet}).\end{array}\]
\end{proposition}
\begin{proof}
For the first statement, it suffices to check when \(\gamma^*F_{\bullet}=\mathbb{Z}\), which follows from \(\H^{*,*}_{\mathrm{M}}(k,\mathbb{Z}/2)/\tau=\mathbf{K}_*^{\mathrm{M}}(k)/2\) by \cite[Theorem 6.17]{V1}. The diffenrential map \(R^{p,q}(F_{\bullet})\to R^{p+2,q+1}(F_{\bullet})\) is given by \((x,y)\mapsto(y,0)\) regarding \(R^{p,q}\) as a subgroup of \(\H^{p,q}_{\mathrm{M}}\oplus\H^{p+2,q+1}_{\mathrm{M}}\). Hence it is easy to get
\[E(\mathrm{MW})_2^{*,*}=\frac{Ker(Sq^2\circ\pi)^{*,*}(\textrm{mod }\tau)}{\pi^{-1}(Im(Sq^2)^{*,*}(\textrm{mod }\tau))},\]
where the latter group is isomorphic to \(E^{*,*}(F_{\bullet})\) by the first statement. Hence we obtain the equality
\[E(\mathrm{MW})_2^{*,*}(F_{\bullet})=E^{*,*}(F_{\bullet}).\]
The statement for \(i\geq3\) follows immediately.
\end{proof}

We note, using the functor $\gamma^*$ and the fact that any Tate motive in \(\mathrm{DM}(k,\mathbb{Z})\) is a direct sum of \(\mathbb{Z}(i)[2i]\), that the condition is satisfied by any Tate MW-motive. This leads to the following definition.

\begin{definition}
For any Tate MW-motive \(A\), we set for $i\geq 2$
\[E_i^{*,*}(A):=E(\mathrm{MW})_i^{*,*}(A)=E(\mathrm{W})_i^{*,*}(A).\]
\end{definition}

\section{MW-Motives of Cellular Varieties over Euclidean Fields}
In this section, we suppose that the base field \(k\) is Euclidean, which is equivalent to \(\mathbf{W}(k)=\mathbb{Z}\). We will also use the localization functor
\[
L:\widetilde{\mathrm{DM}}(k,\mathbb{Z})\longrightarrow\widetilde{\mathrm{DM}}(k,\mathbb{Z})[\eta^{-1}]
\]
which can be seen as the left adjoint of the inclusion $R:\widetilde{\mathrm{DM}}(k,\mathbb{Z})[\eta^{-1}]\to \widetilde{\mathrm{DM}}(k,\mathbb{Z})$ of the full subcategory of $\eta$-local objects. We then have an adjunction
\[
L:\widetilde{\mathrm{DM}}(k,\mathbb{Z})\leftrightarrows \widetilde{\mathrm{DM}}(k,\mathbb{Z})[\eta^{-1}]:R
\]
and we will write $\mathbb{Z}_{\eta}$ for the image of $\mathbb{Z}$ under this functor. We note from \cite[Chapter 6, 4.1.3]{BCDFO} that $\mathbb{Z}_{\eta}$ represents Witt cohomology. 
\begin{definition}
For every \(0\neq p\in\mathbb{N}\), we set
\[
\mathbf{p}=\sum_{i=0}^{p-1}\epsilon^i\in\mathbf{GW}(k)=\mathbf{K}_0^{\mathrm{MW}}(k).
\]
\end{definition}

Since \(\epsilon\eta=\eta\), we have \(\mathbf{p}\eta=p\eta\in\mathbf{K}_{-1}^{\mathrm{MW}}(k)\).

\begin{proposition}\label{wittpart}
Suppose that \(p,q\in\mathbb{Z}\) are odd and that \(l=2^tp\) for some $t\in\mathbb{N}$. The following assertions hold true:
\begin{enumerate}
\item The MW-motive \(\mathbb{Z}/\mathbf{p}\) is the only MW-motive \(A\), up to isomorphism, with \(A/\eta=0\) and \(L(A)=\mathbb{Z}_{\eta}/p\).
\item For any mixed Tate MW-motive \(A\), \(L(A)\) is a direct sum of \(\mathbb{Z}_{\eta}/l[i]\).
\item We have
\[\mathbb{Z}/l\eta=\mathbb{Z}/2^t\eta\oplus\mathbb{Z}/\mathbf{p}.\]
\item If \((p,q)=1\), then
\[
\mathbb{Z}/\mathbf{pq}=\mathbb{Z}/\mathbf{p}\oplus\mathbb{Z}/\mathbf{q}\text{ and }\mathbb{Z}/\mathbf{p}\otimes\mathbb{Z}/\mathbf{q}=0.
\]
\item If \(p\) is an odd prime and \(e_1\leq e_2\) are (positive) integers, we have
\begin{eqnarray*}
\mathbb{Z}/\mathbf{p}^{e_1}\otimes\mathbb{Z}/\mathbf{p}^{e_2} & = &\mathbb{Z}/\mathbf{p}^{e_1}[1]\oplus\mathbb{Z}/\mathbf{p}^{e_1}; \\
\mathbb{Z}/2^{e_1}\eta\otimes\mathbb{Z}/2^{e_2}\eta & = &\mathbb{Z}/2^{e_1}\eta(1)[2]\oplus\mathbb{Z}/2^{e_1}\eta.
\end{eqnarray*}
\end{enumerate}
\end{proposition}

\begin{proof}
We argue as follows:
\begin{enumerate}
\item We may see $\mathbf{p}$ as an endomorphism of $\mathbb{Z}$ in $\widetilde{\mathrm{DM}}(k,\mathbb{Z})$, and correspondingly $\gamma^*(\mathbf{p})$ (resp. $L(\mathbf{p})$) as an endomorphism of $\mathbb{Z}$ in $\mathrm{DM}(k,\mathbb{Z})$ (resp. $\widetilde{\mathrm{DM}}(k,\mathbb{Z})[\eta^{-1}]$). We then have \(\gamma^*(\mathbf{p})=1\) and  \(L(\mathbf{p})=p\); these two properties determining uniquely \(\mathbf{p}\). Thus
\[
\mathbb{Z}/\mathbf{p}\otimes\mathbb{Z}/\eta=0
\]
and then \(\mathbb{Z}/\mathbf{p}\) has the properties described. It follows also that $\mathbb{Z}/\mathbf{p}$ is $\eta$-local and consequently 
\[
[A,\mathbb{Z}/\mathbf{p}]_{\widetilde{\mathrm{DM}}(k,\mathbb{Z})}=[L(A),\mathbb{Z}_{\eta}/p]_{\widetilde{\mathrm{DM}}(k,\mathbb{Z})[\eta]^{-1}}
\]
for any \(A\in\widetilde{\mathrm{DM}}(k,\mathbb{Z})\).
If we suppose that some object \(B\) satisfies the above condition, we have
\[
[A,B]_{\widetilde{\mathrm{DM}}(k,\mathbb{Z})}=[L(A),L(B)]_{\widetilde{\mathrm{DM}}(k,\mathbb{Z})[\eta^{-1}]}=[L(A),\mathbb{Z}_{\eta}/p]_{\widetilde{\mathrm{DM}}(k,\mathbb{Z})[\eta^{-1}]}
\]
and therefore \(B=\mathbb{Z}/\mathbf{p}\) by Yoneda.
\item First, we observe that \([\mathbb{Z}_{\eta}[i],L(A)]_{\widetilde{\mathrm{DM}}(k,\mathbb{Z})[\eta^{-1}]}\) is a finitely generated abelian group since \(A\) comes from finite steps of extensions of \(\mathbb{Z}[i]\) and $L$ is exact. So there is a map
\[\varphi:L(A)\longrightarrow\oplus_{l,i}\mathbb{Z}_{\eta}/l[i]\]
such that \([\mathbb{Z}_{\eta}[i],\varphi]_{\widetilde{\mathrm{DM}}(k,\mathbb{Z})[\eta^{-1}]}\) is an isomorphism for every \(i\). So \(\varphi\) is an isomorphism since \(A\) is mixed Tate. Here, we note that
\[
[\mathbb{Z}_{\eta}[i],\mathbb{Z}_{\eta}/l]_{\widetilde{\mathrm{DM}}(k,\mathbb{Z})[\eta^{-1}]}=\begin{cases}\mathbb{Z}/l&i=0,\\0&i\neq 0,\end{cases}
\]
since $\mathbb{Z}_{\eta}$ represents Witt cohomology.
\item Since \(\mathbf{p}\eta=p\eta\), we may use the composite $\mathbb{Z}(1)[1]\xrightarrow{\mathbf{p}}\mathbb{Z}(1)[1]\xrightarrow{2^t\eta}\mathbb{Z}$ and the octahedron axiom to get a distinguished triangle
\[
\mathbb{Z}/\mathbf{p}(1)[1]\longrightarrow\mathbb{Z}/l\eta\longrightarrow\mathbb{Z}/2^t\eta\longrightarrow\mathbb{Z}/\mathbf{p}(1)[2]
\]
and the statement follows from (1), which shows that $[\mathbb{Z}/2^t\eta,\mathbb{Z}/\mathbf{p}(1)[2]]_{\widetilde{\mathrm{DM}}(k,\mathbb{Z})}=0$ and that $\mathbb{Z}/\mathbf{p}(1)[1]\simeq \mathbb{Z}/\mathbf{p}$.
\item This follows directly from (1).
\item We only prove the last statement, the proof of the first one being straightforward. The distinguished triangle
\[
\mathbb{Z}(2)[2]\xrightarrow{2^{e_1}\eta} \mathbb{Z}(1)[1]\to \mathbb{Z}/2^{e_1}\eta(1)[1]\to \mathbb{Z}(2)[3]
\]
yields an exact sequence
\[
[\mathbb{Z}(2)[3],\mathbb{Z}]_{\widetilde{\mathrm{DM}}}\longrightarrow[\mathbb{Z}/2^{e_1}\eta(1)[1],\mathbb{Z}]_{\widetilde{\mathrm{DM}}}\longrightarrow[\mathbb{Z}(1)[1],\mathbb{Z}]_{\widetilde{\mathrm{DM}}}\xrightarrow{2^{e_1}\eta}[\mathbb{Z}(2)[2],\mathbb{Z}]_{\widetilde{\mathrm{DM}}}
\]
and we have 
\[
[\mathbb{Z}(i)[j],\mathbb{Z}]_{\widetilde{\mathrm{DM}}(k,\mathbb{Z})}=[\mathbb{Z},\mathbb{Z}(-i)[-j]]_{\widetilde{\mathrm{DM}}(k,\mathbb{Z})}=\mathrm{H}^{i-j}(k,\mathbf{W})
\]
for any $i,j>0$. It follows that $[\mathbb{Z}(2)[3],\mathbb{Z}]_{\widetilde{\mathrm{DM}}(k,\mathbb{Z})}=0$ and that the morphism
\[
\mathbf{W}(k)=[\mathbb{Z}(1)[1],\mathbb{Z}]_{\widetilde{\mathrm{DM}}(k,\mathbb{Z})}\xrightarrow{2^{e_1}\eta}[\mathbb{Z}(2)[2],\mathbb{Z}]_{\widetilde{\mathrm{DM}}(k,\mathbb{Z})}=\mathbf{W}(k)
\]
is injective. Therefore, $[\mathbb{Z}/2^{e_1}\eta(1)[1],\mathbb{Z}]_{\widetilde{\mathrm{DM}}(k,\mathbb{Z})}=0$. Next, the triangle
\[
\mathbb{Z}\to \mathbb{Z}/2^{e_1}\eta\xrightarrow{d} \mathbb{Z}(1)[2]\xrightarrow{2^{e_1}\eta}\mathbb{Z}[1]
\]
yields this time an exact sequence
\[
[\mathbb{Z}[1],\mathbb{Z}[1]]_{\widetilde{\mathrm{DM}}}\xrightarrow{2^{e_1}\eta}[\mathbb{Z}(1)[2],\mathbb{Z}[1]]_{\widetilde{\mathrm{DM}}}\longrightarrow[\mathbb{Z}/2^{e_1}\eta,\mathbb{Z}[1]]_{\widetilde{\mathrm{DM}}}\longrightarrow[\mathbb{Z},\mathbb{Z}[1]]_{\widetilde{\mathrm{DM}}}
\]
from which we draw that $[\mathbb{Z}/2^{e_1}\eta,\mathbb{Z}[1]]_{\widetilde{\mathrm{DM}}(k,\mathbb{Z})}=\mathbb{Z}/2^{e_1}$ using that $[\mathbb{Z},\mathbb{Z}[1]]_{\widetilde{\mathrm{DM}}(k,\mathbb{Z})}=0$, while $[\mathbb{Z}[1],\mathbb{Z}[1]]_{\widetilde{\mathrm{DM}}(k,\mathbb{Z})}=\mathbf{K}_0^{\mathrm{MW}}(k)$ and $[\mathbb{Z}(1)[2],\mathbb{Z}[1]]_{\widetilde{\mathrm{DM}}(k,\mathbb{Z})}=\mathbf{W}(k)=\mathbb{Z}$, generated by $\eta$. Applying $[\mathbb{Z}/2^{e_1}\eta(1)[1],-]$ to the exact triangle above, we obtain 
\[
[\mathbb{Z}/2^{e_1}\eta(1)[1],\mathbb{Z}/2^{e_1}\eta]_{\widetilde{\mathrm{DM}}(k,\mathbb{Z})}=\mathbb{Z}/2^{e_1}
\]
and it follows that the first morphism in the distinguished triangle
\[
\mathbb{Z}/2^{e_1}\eta(1)[1]\xrightarrow{2^{e_2}\eta}\mathbb{Z}/2^{e_1}\eta\longrightarrow\mathbb{Z}/2^{e_1}\eta\otimes\mathbb{Z}/2^{e_2}\eta\longrightarrow \mathbb{Z}/2^{e_1}\eta(1)[2]
\]
is trivial, providing a splitting.
\end{enumerate}
\end{proof}

\begin{proposition}\label{leta}
Suppose that \(l=2^ts\) with \(s\) odd and that \(p,q\in\mathbb{Z}\). We have
\[[\mathbb{Z}/l\eta,\mathbb{Z}(q)[p]]_{\widetilde{\mathrm{DM}}(k,\mathbb{Z})}=\begin{cases}\mathrm{H}_{\mathrm{M}}^{p,q}(k,\mathbb{Z})\oplus \mathrm{H}^{p-2,q-1}_{\mathrm{M}}(k,\mathbb{Z})& p\neq q,q+1,\\2\mathbf{K}_q^{\mathrm{M}}(k)\oplus \mathrm{H}^{p-2,q-1}_{\mathrm{M}}(k,\mathbb{Z})&p=q,\\\mathbf{I}^q(k)/s\mathbf{I}^q(k)\oplus\mathbf{I}^{q-1}(k)/2^t\mathbf{I}^q(k)\oplus2\mathbf{K}_{q-1}^{\mathrm{M}}(k)&p=q+1,\\\end{cases}\]
and
\[[\mathbb{Z}/l\eta,\mathbb{Z}(q)[p]]_{\mathrm{WDM}(k,\mathbb{Z})}=\begin{cases} \mathrm{H}_{\mathrm{M}}^{p,q}(k,\mathbb{Z}/2)\oplus  \mathrm{H}^{p-2,q-1}_{\mathrm{M}}(k,\mathbb{Z}/2)&p\neq q,q+1.\\ \mathrm{H}^{p-2,q-1}_{\mathrm{M}}(k,\mathbb{Z}/2)&p=q.\\\mathbf{I}^q(k)/s\mathbf{I}^q(k)\oplus\mathbf{I}^{q-1}(k)/2^t\mathbf{I}^q(k)&p=q+1.\\\end{cases}\]
\end{proposition}
\begin{proof}
We do the first case,the second being much easier. By definition, we have a long exact sequence
\[\ldots\xrightarrow{l\eta}\mathrm{H}^{p-2,q-1}_{\mathrm{MW}}(k,\mathbb{Z})\longrightarrow [\mathbb{Z}/l\eta,\mathbb{Z}(q)[p]]_{\widetilde{\mathrm{DM}}(k,\mathbb{Z})}\longrightarrow \mathrm{H}^{p,q}_{\mathrm{MW}}(k,\mathbb{Z})\xrightarrow{l\eta}\ldots \]
from which the proof follows if \(p\neq q+1\) using the argument in the proof of \cite[Proposition 2.2]{Y2}.

Now suppose \(p=q+1\). By Proposition \ref{wittpart}, it suffices to consider the cases when \(s=1\) or \(t=0\). Since \(k\) is Euclidean, there is a decomposition
\[\mathbf{K}_q^{\mathrm{M}}(k)=2\mathbf{K}_q^{\mathrm{M}}(k)\oplus\mathbb{Z}/2\cdot[-1]^q\]
where \(2\mathbf{K}_q^{\mathrm{M}}(k)\) is \(2\)-divisible. There is a commutative diagram with exact rows
\[
	\xymatrix
	{
		\mathbf{K}_q^{\mathrm{MW}}(k)\ar[r]^{l\eta}\ar[d]	&\mathbf{K}_{q-1}^{\mathrm{MW}}(k)\ar[r]\ar[d]_-f	&[\mathbb{Z}/l\eta,\mathbb{Z}(q)[p]]_{\widetilde{\mathrm{DM}}(k,\mathbb{Z})}\ar[r]\ar@{-->}[d]_-{\overline f}	&0\\
		0\ar[r]														&\mathbf{K}_{q-1}^{\mathrm{M}}(k)\ar@{=}[r]		&\mathbf{K}_{q-1}^{\mathrm{M}}(k)\ar[r]										&0,
	}
\]
where $f$ is the projection. Since $f$ is onto, \(\overline f\) has the same property. By the Snake lemma we get an exact sequence
\[\mathbf{K}_q^{\mathrm{MW}}(k)\xrightarrow{l\eta}\mathbf{I}^q(k)\longrightarrow \ker(\overline f)\longrightarrow 0.\]
So \(\ker(\overline f)=\mathbf{I}^q(k)/l\mathbf{I}^q(k)\) and there is an exact sequence
\[0\longrightarrow\mathbf{I}^q(k)/l\mathbf{I}^q(k)\longrightarrow[\mathbb{Z}/l\eta,\mathbb{Z}(q)[p]]_{\widetilde{\mathrm{DM}}(k,\mathbb{Z})}\longrightarrow\mathbf{K}_{q-1}^{\mathrm{M}}(k)\longrightarrow 0.\]
Considering this time the epimorphism $\mathbf{K}_{q-1}^{\mathrm{MW}}\to\mathbf{I}^{q-1}$, we obtain a diagram with exact rows
\[
	\xymatrix
	{
		0\ar[r]	&l\mathbf{I}^q(k)\ar[r]\ar@{=}[d]	&\mathbf{K}_{q-1}^{\mathrm{MW}}(k)\ar[r]\ar[d]	&[\mathbb{Z}/l\eta,\mathbb{Z}(q)[p]]_{\widetilde{\mathrm{DM}}(k,\mathbb{Z})}\ar[r]\ar[d]_g	&0\\
		0\ar[r]	&l\mathbf{I}^q(k)\ar[r]					&\mathbf{I}^{q-1}(k)\ar[r]			&\mathbf{I}^{q-1}(k)/l\mathbf{I}^q(k)\ar[r]					&0
	}
\]
where \(g\) is surjective. The composite
\[2^q\mathbb{Z}/l2^q\mathbb{Z}=\mathbf{I}(k)^q/l\mathbf{I}(k)^q\longrightarrow[\mathbb{Z}/l\eta,\mathbb{Z}(q)[p]]_{\widetilde{\mathrm{DM}}(k,\mathbb{Z})}\xrightarrow{g}\mathbf{I}(k)^{q-1}/l\mathbf{I}(k)^q=2^{q-1}\mathbb{Z}/l2^q\mathbb{Z}\]
is split injective if \(l\) is odd. We have an isomorphism \(\ker(g)=2\mathbf{K}_{q-1}^M(k)\) and consequently we get an exact sequence
\[0\longrightarrow2\mathbf{K}_{q-1}^M(k)\longrightarrow[\mathbb{Z}/l\eta,\mathbb{Z}(q)[p]]_{\widetilde{\mathrm{DM}}(k,\mathbb{Z})}\longrightarrow\mathbf{I}(k)^{q-1}/l\mathbf{I}(k)^q\longrightarrow 0,\]
which splits since \(l=2^t\) and \(2\mathbf{K}_{q-1}^M(k)\) is \(2\)-divisible.
\end{proof}
\begin{proposition}\label{faithful}
Suppose that \(f:A\longrightarrow B\) is a map between Tate MW-motives having up to \(\eta^r\)-torsions. If \(\mathrm{CH}^*(f)\) and \(\mathrm{H}^*(f,\mathbf{W})\) are isomorphisms, then \(f\) is an isomorphism.
\end{proposition}
\begin{proof}
The hypothesis implies that \(\mathrm{H}^{*,*}_{\mathrm{M}}(f,\mathbb{Z})\), \(\mathrm{H}^{*,*}_{\mathrm{M}}(f,\mathbb{Z}/2)\) are isomorphisms since $\gamma^*(A)$ and $\gamma^*(B)$ are Tate. This implies that \(E^{*,*}_i(f)\) are isomorphisms for all \(i\geq 1\) by Proposition \ref{bockstein}. Since both of them have up to \(\eta^r\)-torsion, in the Bockstein spectral sequences of both \(A\) and \(B\), the corresponding quotients \(Z_{r+1}/B_{r+1}\) are isomorphic. Since \(Z_{i}/Z_{i+1}\) is the image of \(E_{i+1}\longrightarrow E_{i+1}\), the quotients \(Z_i/B_{r+1}\) for \(A\) and \(B\) are isomorphic for all \(i\). Moreover, we have \(D_r\cap \ker(\eta^{\infty})=0\). Applying \eqref{filtration} for \(0\leq n\leq r\) and the hypothesis on \(H^*(f,\textbf{W})\), we see that \(\mathrm{H}^{*,*}_{\mathrm{MW}}(f,\mathbb{Z})\) is an isomorphism. So \(f\) is an isomorphism by the property of Tate objects.
\end{proof}
\begin{lemma}\label{linear}
Suppose that \(f:\mathbb{Z}^{\oplus n}=V\longrightarrow W=\mathbb{Z}/2^{\oplus m}\) is a surjective homomorphism of $\mathbb{Z}$-modules. The following assertions are verified:
\begin{enumerate}
\item There is a decomposition \(V=V_1\oplus V_2\) such that $W=V_1/2V_1$ and $\ker(f)=2V_1\oplus V_2$.
\item If \(W=W_1\oplus W_2\), there is a decomposition \(V=V_1\oplus V_2\) with $W_1=V_1/2V_1$ and $f(V_2)=W_2$.
\end{enumerate}
\end{lemma}
\begin{proof}
\begin{enumerate}
\item We have a homomorphism \(\varphi\colon\mathbb{Z}^n=\ker(f)\longrightarrow V\). By Smith normal form, there are basis \(\{e_i\},\{e'_i\},1\leq i\leq n\) of \(\ker(f)\) and \(V\) respectively such that
\[
\varphi(e_s)=\begin{cases}2e'_s&s\leq m.\\e'_s&s>m.\end{cases}
\]
Set \(V_1=span\{e'_1,\cdots,e'_{m}\}\) and \(V_2=span\{e'_{m+1},\cdots,e'_n\}\) to conclude.
\item Denote by \(V=V'\oplus V''\) the decomposition given by (1) and by \(f'\) the composite \(V'\to V\xrightarrow{f}W\). Then
\[W_1=f'^{-1}(W_1)/2, f(f'^{-1}(W_2)\oplus V'')=W_2.\]
\end{enumerate}
\end{proof}
\begin{theorem}\label{tate0}
Over an arbitrary field \(k\), suppose that \(A\) is a Tate MW-motive with \(\mathrm{H}^*(A,\mathbf{W})=0\). Then, it is a direct sum of motives of the form \(\mathbb{Z}/\eta(i)[2i]\).
\end{theorem}
\begin{proof}
Since \(\mathrm{H}^*(A,\mathbf{W})=0\), we have \([A,\mathbb{Z}(i)[2i+1]]_{\widetilde{\mathrm{DM}}(k,\mathbb{Z})}=0\) for all \(i\in\mathbb{N}\). Since \(A\) is Tate, we conclude that
\[[A(1)[2],A[1]]_{\widetilde{\mathrm{DM}}(k,\mathbb{Z})}=0.\]
In particular, \(\eta\otimes Id_A=0\). So there is a decomposition
\[A/\eta=A\oplus A(1)[2],\]
which implies that the map
\[\mathrm{H}^{*,*}_{\mathrm{MW}}(A,\mathbb{Z})\longrightarrow E(MW)_1^{*,*}(A)\]
is injective. So, \(A\) has up to \(\eta\)-torsion by Proposition \ref{1}.
 
The ring \(\mathrm{CH}^*(A)\) is free since \(\gamma^*(A)\) is Tate. By Proposition \ref{bockstein} and the discussion above, we have
\[\ker(Sq^2\circ\pi)^{2*,*}(A)=\pi^{-1}(\mathrm{Im}(Sq^2))^{2*,*}(A).\]
We have a commutative diagram for \(A\)
\[
	\xymatrix
	{
		\pi^{-1}(\mathrm{Im}(Sq^2))^{2n,n}\ar[r]^-{\varphi}\ar[dr]_{\pi}	&\mathrm{CH}^{n-1}/\ker(Sq^2\circ\pi)^{2n-2,n-1}\ar[d]_{Sq^2\circ\pi}^{\cong}\\
																						&\mathrm{Im}(Sq^2)^n.
	}
\]
By Lemma \ref{linear}, there is a basis \(\{y_{n,i},z_{n,j}\}\) of \(\mathrm{CH}^n(A)\) such that
\[\sum_i\alpha_iy_{n,i}+\sum_j\beta_jz_{n,j}\in \ker(Sq^2\circ\pi)^{2n,n}\]
if and only if \(\alpha_i\) are even for every \(i\). Then \(\pi^{-1}(\mathrm{Im}(Sq^2))^{2n,n}/(\oplus_i2\mathbb{Z}y_{n,i})=\oplus_j\mathbb{Z}z_{n,j}\) is free \(\mathbb{Z}\)-module, the rank of which is equal to that of \(\mathrm{CH}^{n-1}/\ker(Sq^2\circ\pi)^{n-1}\) as a \(\mathbb{Z}/2\)-module. So there is a basis \(\{w_{n,j}\}\) of \(\pi^{-1}(\mathrm{Im}(Sq^2))^{2n,n}/(\oplus_i2\mathbb{Z}y_{n,i})\) satisfying
\[\varphi(w_{n,j})=\pi(y_{n-1,j}),\]
and we obtain classes \(t_{n,i}=(y_{n,i},w_{n+1,i})\in E_1^{2n,n}(A)\) satisfying
\[
\ker(Sq^2\circ\pi)^n(A)=(\oplus_i2\mathbb{Z}y_{n,i})\oplus(\oplus_j\mathbb{Z}w_{n,j}).
\]
Hence
\[\mathrm{CH}^n(A)=\left(\oplus_i\mathbb{Z}y_{n,i}\right)\oplus\left(\oplus_j\mathbb{Z}w_{n,j}\right).\]
So
\[A\xrightarrow{t_{n,i}}\oplus_{n,i}\mathbb{Z}/\eta(n)[2n]\]
induces an isomorphism on \(\mathrm{CH}^*\) and \(\mathrm{H}^*(-,\mathbf{W})\), which concludes the proof by Proposition \ref{faithful}.
\end{proof}

\begin{proposition}\label{bounded}
The following assertions are verified:
\begin{enumerate}
\item Every Tate MW-motive \(A\) has up to \(\eta^r\)-torsion for some \(r\). Hence that hypothesis in Proposition \ref{faithful} is unnecessary.
\item For any Tate MW-motive \(A\), there is a map
\[A\longrightarrow\oplus_i\mathbb{Z}(i)[2i]\oplus\oplus_{i,j}\mathbb{Z}/j\eta(i)[2i]\]
inducing isomorphisms on \(\mathrm{H}^*(-,\mathbf{W})\) and \(E_i^{*,*}(-)\) for \(i\geq 2\).
\item For any Tate MW-motive \(A\), we have
\[\mathbf{I}^{2q-p}(k)\mathrm{H}^{p-q}(A,\mathbf{W})=\mathrm{H}^{p,q}_{\eta}(A).\]
\end{enumerate}
\end{proposition}
\begin{proof}
We proceed as follows:
\begin{enumerate}
\item By Proposition \ref{wittpart}, there is a map
\[h[-1]:C[-1]=\oplus_{p,r,i}\mathbb{Z}/\mathbf{p}^r[i]\oplus\oplus_{r,i}\mathbb{Z}/2^r\eta(i)[2i-1]\longrightarrow A\]
such that \(\mathrm{H}^*(C(h[-1]),\mathbf{W})=0\) and \(L(h[-1])\) is an isomorphism. So we have a distinguished triangle
\[
A\xrightarrow{f}B\longrightarrow C\xrightarrow{h}A[1]
\]
where \(B\) is a direct sum of \(\mathbb{Z}/\eta(i)[2i]\) by Theorem \ref{tate0} and \(C\) has up to \(\eta^r\)-torsion for some \(r\) by Proposition \ref{leta}. For any \(x\in \mathrm{H}^{*,*}_{\mathrm{MW}}(A,\mathbb{Z})\) with \(\eta^nx=0\) for some \(n\geq r\), we have \(h^*(\eta^nx)=0\) so \(h^*(\eta^rx)=0\). So, there is \(y\in \mathrm{H}^{*,*}_{\mathrm{MW}}(B,\mathbb{Z})\) such that \(f^*(y)=\eta^rx\). But \(\eta y=0\), yielding \(\eta^{r+1}x=0\).
\item We claim that
\[
E^{p,q}_2(A)=\begin{cases}\mathrm{H}^{p-q}(A,\mathbf{W}/2)&p\leq 2q.\\0&p>2q.\end{cases}
\]
If \(p>2q\), \(\mathrm{H}^{p,q}_{\mathrm{M}}(A,\mathbb{Z}/2)=0\) so the claim follows. Let \(k'\) be the field \(k\sqrt{-1}\). If \(p=2q\), then \(\mathrm{CH}^*(A)=\mathrm{CH}^*(A\times_kk')\) so \(E^{p,q}_2(A)=E^{p,q}_2(A\times_kk')\). Now, \(A\times_kk'\) is a direct sum of \(\mathbb{Z}(i)[2i]\) and \(\mathbb{Z}/\eta(i)[2i]\) since \(\mathbf{W}(k')=\mathbb{Z}/2\). So we have
\[E^{p,q}_2(A\times_kk')=\mathrm{H}^{p-q}(A\times_kk',\mathbf{W})=\mathrm{H}^{p-q}(A,\mathbf{W}/2)\]
by \cite[Theorem 2.22]{Y2} and the universal coefficient theorem. If \(p<2q\), we have an exact sequence
\[0\longrightarrow \mathrm{H}_{\mathrm{M}}^{p-1,q-1}(A,\mathbb{Z}/2)\xrightarrow{\rho}\mathrm{H}_{\mathrm{M}}^{p,q}(A,\mathbb{Z}/2)\longrightarrow \mathrm{H}_{\mathrm{M}}^{p,q}(A\times_kk',\mathbb{Z}/2)\longrightarrow 0.\]
They are free \(\mathbb{Z}/2[\tau]\)-modules hence the sequence is exact after taking quotients by \(\tau\). We will prove the claim by induction on \(2q-p\), the case \(2q=p\) being established. One easily shows (3) for \(A\times_kk'\) so by \textit{loc. cit.} we have
\[E^{p,q}_2(A\times_kk')=0, p<2q.\]
Applying the Snake lemma to the sequence above gives
\[E^{p,q}_2(A)=E^{p-1,q-1}_2(A), p<2q.\]
Then the claim follows by induction hypothesis.

In the context above, since \([C,A/\eta[1]]_{\widetilde{\mathrm{DM}}(k,\mathbb{Z})}=0\), there is a map \(f:C\longrightarrow A(1)[2]\) which factors through the boundary map \(C\longrightarrow A[1]\), which induces an isomorphism on \(\mathrm{H}^*(-,\mathbf{W})\), hence on \(\mathrm{H}^*(-,\mathbf{W}/2)\). So it induces an isomorphism
\[E^{p,q}_2(A)=E^{p+2,q+1}_2(C)\]
by the claim before. Hence so does each \(E_i\)-page for \(i\geq 2\).
\item In the context above, we have
\[\mathrm{H}^{p,q}_{\eta}(A)=\mathrm{H}^{p+2,q+1}_{\eta}(C).\]
So it suffices to check the statement for \(C\), which follows from Proposition \ref{leta}.
\end{enumerate}
\end{proof}
Denote by \(h^{*,*}=\textrm{H}^{*,*}_{\textrm{M}}(k,\mathbb{Z}/2)=\mathbb{Z}/2[\rho,\tau]\). It naturally acts on \(E^{*,*}_i(A),i\geq 2\) for every Tate MW-motive \(A\), where the \(\tau\) acts trivially by Proposition \ref{bockstein}.
\begin{proposition}\label{higher}
Suppose that \(i\geq 2,j\in\mathbb{N}\) and that \(p,q\in\mathbb{Z}\). The groups \(u,v\) of \(E^{0,0}_i(\mathbb{Z}/2^j\eta) (\textrm{if }i\leq j+1,j>0)\) and \(E^{2,1}_i(\mathbb{Z}/2^j\eta)(\textrm{if }j>0)\) are isomorphic to \(\mathbb{Z}/2\cdot u\) and \(\mathbb{Z}/2\cdot v\), respectively. We have
\[E^{p,q}_i(\mathbb{Z}/2^j\eta)=\begin{cases}\mathbb{Z}/2\cdot\rho^qu&\textrm{if }j>0,i\leq j+1,p=q,q\geq 0\\\mathbb{Z}/2\cdot\rho^{q-1}v&\textrm{if }j>0,i\leq j+1,p=q+1,q\geq 1\\\mathbb{Z}/2\cdot\rho^{q-1}v&\textrm{if }j>0,p=q+1,0<q<j+1\\0&\textrm{else}\end{cases}.\]
The differential map
\[E^{p,q}_i(\mathbb{Z}/2^j\eta)\longrightarrow E^{p+i+1,q+i}_i(\mathbb{Z}/2^j\eta)\]
is zero unless \(p=q\geq 0\) and \(i=j+1\).
\end{proposition}
\begin{proof}
We compute the Bockstein map
\[h^{p-2,q-1}\oplus h^{p-4,q-2}=H^{p-2,q-1}_M(\mathbb{Z}/2^j\eta,\mathbb{Z}/2)\xrightarrow{\beta} H^{p,q}_W(\mathbb{Z}/2^j\eta,\mathbb{Z}).\]
We have a map between distinguished triangles
\[
	\xymatrix
	{
		\mathbb{Z}(1)[1]\ar[r]^-{2^j\eta}\ar[d]_{1+\epsilon}	&\mathbb{Z}\ar[r]\ar@{=}[d]	&\mathbb{Z}/2^j\eta\ar[r]\ar[d]_{f}	&\mathbb{Z}(1)[2]\ar[d]_{1+\epsilon}\\
		\mathbb{Z}(1)[1]\ar[r]^-{2^{j-1}\eta}							&\mathbb{Z}\ar[r]					&\mathbb{Z}/2^{j-1}\eta\ar[r]^d				&\mathbb{Z}(1)[2]
	},
\]
which induces a commutative diagram
\[
	\xymatrix
	{
		h^{p-4,q-2}\ar[r]^c\ar[d]_{d^*}								&H^{p-2,q-1}_W(k,\mathbb{Z})\ar[d]_{d^*}\\
		h^{p-2,q-1}\oplus h^{p-4,q-2}\ar[r]\ar[d]_{a:=f^*}	&H^{p,q}_W(\mathbb{Z}/2^{j-1}\eta,\mathbb{Z})\ar[d]_{b:=f^*}\\
		h^{p-2,q-1}\oplus h^{p-4,q-2}\ar[r]						&H^{p,q}_W(\mathbb{Z}/2^{j}\eta,\mathbb{Z})
	}.
\]
We have that \(a(x,y)=(x,0)\) since \(1+\epsilon=0\) in \(\textbf{K}_0^M(k)\) and that \(c=0\) since \(Sq^2\) acts trivially on \(h^{*,*}\). By Proposition \ref{leta} and \(1+\epsilon=2\) in \(\textbf{W}(k)\), we have
\[b=\begin{cases}a&p\neq q+1,q\\0&p=q\\2^{q-1}\mathbb{Z}/2^{q+j-1}\xrightarrow{2}2^{q-1}\mathbb{Z}/2^{q+j}&p=q+1\end{cases}.\]
So if \(p\neq q+1,q\), we have \(\beta(x,y)=(0,2^jx)\). If \(p=q\), we have \(\beta(x,y)=2^jx\). If \(p=q+1\), by \cite[Proposition 2.12]{Y2}, the \(\beta\) is the map
\[h^{p-2,q-1}\oplus h^{p-4,q-2}=2^{q-1}\mathbb{Z}/2^q\oplus (2^{q-3}\mathbb{Z}/2^{q-2})\tau\xrightarrow{(2^j,0)}2^{q-1}\mathbb{Z}/2^{q+j}.\]

Let us prove the statement by induction on \(i\). If \(i=2\), using Proposition \ref{bockstein}, we see that it suffices to work out \(Sq^2\) modulo \(\tau\). If \(p=q+1\), it is identified with
\[\rho^{q-1}h^{0,0}=h^{p-2,q-1}\xrightarrow{2^j}2^{q-1}\mathbb{Z}/2^{q+j}\longrightarrow 2^{q-1}\mathbb{Z}/2^q,\] 
which is zero if \(j\neq 0\), otherwise it is surjective. Suppose \((x,y)\in h^{p-2,q-1}/\tau\oplus h^{p-4,q-2}/\tau\), which is zero if \(p\neq q+1,q+2\). If \(p=q+2\), the target of \(Sq^2\) vanishes and
\[H^{p-2,q-1}_M(\mathbb{Z}/2^j\eta,\mathbb{Z}/2)/\tau=h^{q-2,q-2}=\rho^{q-2}h^{0,0}.\]
So
\[Sq^2:H^{p-2,q-1}_M(\mathbb{Z}/2^j\eta,\mathbb{Z}/2)/\tau\longrightarrow H^{p,q}_M(\mathbb{Z}/2^j\eta,\mathbb{Z}/2)/\tau\]
is nonzero, hence an isomorphism, if and only if \(p=q+1,j=0,q\geq 1\). This computes \(E^{*,*}_2(\mathbb{Z}/2^j\eta)\), thus proving the statement for the differential maps when \(i=2\).

Suppose the statement is true for \(i=i_0\geq 2\). So
\[E^{p,q}_{i+1}(\mathbb{Z}/2^j\eta)=E^{p,q}_i(\mathbb{Z}/2^j\eta)\]
if either \(p\neq q,q+1\) or \(p=q+1\) and \(i\neq j+1\) or \(q<0\). If \(p=q+1\) and \(i=j+1\), consider the complex
\[E^{q-j-1,q-j-1}_{j+1}(\mathbb{Z}/2^j\eta)\longrightarrow E^{q+1,q}_{j+1}(\mathbb{Z}/2^j\eta)\longrightarrow 0,\]
which is identified with
\[2^{q-j-1}\mathbb{Z}/2^q\xrightarrow{2^j}2^{q-1}\mathbb{Z}/2^q\longrightarrow 0\]
if \(q\geq j+1\) and with
\[0\longrightarrow2^{q-1}\mathbb{Z}/2^q\longrightarrow 0\]
otherwise. So we have by induction
\[
E^{q+1,q}_{j+2}(\mathbb{Z}/2^j\eta)=\begin{cases}\mathbb{Z}/2\cdot\rho^{q-1}v&0<q<j+1.\\0&q\geq j+1.\end{cases}
\]
Next suppose \(p=q\geq 0\) and \(i=j+1\). We consider the differential map
\[E^{q,q}_{j+1}(\mathbb{Z}/2^j\eta)\longrightarrow\eta^{j}H^{q+j+2,q+j+1}(\mathbb{Z}/2^j\eta,\mathbb{Z})\longrightarrow E^{q+j+2,q+j+1}_{j+1}(\mathbb{Z}/2^j\eta)\]
with zero boundary. By the discussion before, this is identified with the composite
\[2^q\mathbb{Z}/2^{q+1}\xrightarrow{2^j}2^q\mathbb{Z}/2^{q+j+1}\longrightarrow2^{q+j}\mathbb{Z}/2^{q+j+1},\]
which is an isomorphism. So
\[E^{q,q}_{j+2}(\mathbb{Z}/2^j\eta)=0,\]
which implies
\[E^{q+1,q}_i(\mathbb{Z}/2^j\eta)=\mathbb{Z}/2\cdot\rho^{q-1}v\]
if \(0<q<j+1\) and \(i\geq j+2\). We conclude the proof.
\end{proof}

\begin{corollary}\label{dim}
Suppose \(A\) that is Tate and that
\[H^n(A,\textbf{W})=\oplus_{j=1}^{\infty}(\mathbb{Z}/2^j)^{\oplus x_{n,j}}\oplus\textrm{odd torsions},\]
where (\(\mathbb{Z}/2^{\infty}:=\mathbb{Z}\)).
We have
\[\oplus_tE_i^{2n+t,n+t}(A)=\oplus_{j=1}^{i-2}(\mathbb{Z}/2[\rho]/\rho^j)^{\oplus x_{n-1,j}}\oplus\mathbb{Z}/2[\rho]^{\oplus_{j=i-1}^{\infty}x_{n-1,j}+x_{n,j}}\]
as \(\mathbb{Z}/2[\rho]\)-modules for \(i\geq 2\). Moreover, the Bockstein
\[\oplus_tE_i^{2n+t,n+t}(A)\longrightarrow\oplus_tE_i^{2n+2+t,n+1+t}(A)\]
is identified with the composite
\[\oplus_tE_i^{2n+t,n+t}(A)\xrightarrow{proj}\mathbb{Z}/2[\rho]^{\oplus_{j=i-1}^{\infty}x_{n,j}}\xrightarrow{\rho^{i-1}}\mathbb{Z}/2[\rho]^{\oplus_{j=i-1}^{\infty}x_{n,j}}\subseteq\oplus_tE_i^{2n+2+t,n+1+t}(A).\]
\end{corollary}
Let us consider the Leibniz formula of the Bockstein spectral sequence. For Tate MW-motives, there is an exterior product on \(E(\textrm{M/2})_1^{*,*}\) defined by \((a,b)\times(c,d)=(a\times c,b\times c+a\times d)\) as in \cite[Proposition 2.14]{Y2}. The map
\[\begin{array}{ccc}E(\textrm{W})_1^{*,*}(A)=H^{*,*}_M(A,\mathbb{Z})\oplus H^{*+2,*}_M(A,\mathbb{Z})&\longrightarrow&H^{*,*}_M(A,\mathbb{Z})\oplus H^{*+2,*+1}_M(A,\mathbb{Z})=E(\textrm{M/2})_1^{*,*}(A)\\(x,y)&\longmapsto&(x,Sq^2(x)+\tau y)\end{array}\]
is injective for every Tate \(A\) since \(\tau\) is a non-zero-divisor. This gives \(E(\textrm{W})_1^{*,*}\) an exterior product with
\[(a,b)\times(c,d)=(a\times c,Sq^1(a)\times Sq^1(c)+a\times d+b\times c)\]
by Cartan formula, with \(\beta_1(x\times y)=x\times\beta_1(y)+\beta_1(x)\times y\) inherited from that of \(E(\textrm{M/2})_1^{*,*}\). Note that the exterior product on \(E(\textrm{MW})_1^{*,*}\) does not satisfy the Leibniz formula.

In general, if the Leibniz formula
\[\beta_i(x\times y)=x\times\beta_i(y)+\beta_i(x)\times y\]
holds for all Tate MW-motives, their \(E_{i+1}^{*,*}\) admits an exterior product automatically.
\begin{proposition}\label{leibniz1}
In the context above, the exterior product
\[E_{i}^{p,q}(\mathbb{Z}/2^j\eta)\times E_{i}^{s,t}(\mathbb{Z}/2^k\eta)\longrightarrow E_{i}^{p+s,q+t}(\mathbb{Z}/2^j\eta\otimes\mathbb{Z}/2^k\eta)\]
satisfies the Leibniz formula.
\end{proposition}
\begin{proof}
We may suppose \(j,k\geq 1\). We claim that
\[E_{i}^{p,q}(\mathbb{Z}/2^j\eta\otimes\mathbb{Z}/2^k\eta)=\begin{cases}\mathbb{Z}/2\cdot\rho^qu\times u&p=q\\\mathbb{Z}/2\cdot\rho^{q-1}(u\times v\oplus v\times u)&p=q+1\\\mathbb{Z}/2\cdot\rho^{q-2}v\times v&p=q+2\\0&\textrm{else}\end{cases},2\leq i\leq j+1,k+1.\]
The dimensions are correct by Proposition \ref{wittpart} and Proposition \ref{higher}.

Suppose \(i=2\). The Leibniz formula has been established. We have
\[H^{p,q}_M(\mathbb{Z}/2^j\eta\otimes\mathbb{Z}/2^k\eta,\mathbb{Z}/2)=h^{p,q}u\times u\oplus h^{p-2,q-1}(u\times v\oplus v\times u)\oplus h^{p-4,q-2}v\times v\]
by K\"unneth formula in \(DM(k,\mathbb{Z}/2)\). By Proposition \ref{higher}, the \(Sq^2=0\) modulo \(\tau\) on the group above, for every \(j,k\geq 1\). So the claim follows. So does the case \(i\leq j+1\) since \(\beta_i=0,i\leq j,k\).

If \(i<j+1\leq k+1\), the Leibniz formula holds since the Bocksteins vanish.

If \(i=j+1\leq k+1\), the map
\[-\times v:E_{j+1}^{p-2,q-1}(\mathbb{Z}/2^j\eta)\longrightarrow E_{j+1}^{p,q}(\mathbb{Z}/2^j\eta\otimes\mathbb{Z}/2^k\eta)\]
is induced by the map \(\mathbb{Z}/2^j\eta\otimes(\mathbb{Z}/2^k\eta\xrightarrow{\partial}\mathbb{Z}(1)[2])\). So we have
\[\beta_{j+1}(u\times v)=\beta_{j+1}\circ(Id\times\partial)^*(u)=(Id\times\partial)^*\circ\beta_{j+1}(u)=\rho^j(Id\times\partial)^*(v)=\rho^jv\times v,\]
similarly for \(v\times u\). Denote by \(p:\mathbb{Z}\longrightarrow\mathbb{Z}/2^j\eta\) the quotient map. It induces a map
\[(Id\otimes p)^*:E^{*,*}_{j+1}(\mathbb{Z}/2^j\eta\otimes\mathbb{Z}/2^k\eta)\longrightarrow E^{*,*}_{j+1}(\mathbb{Z}/2^k\eta)\]
such that \((Id\otimes p)^*(x\times u)=x,(Id\otimes p)^*(x\times v)=0\), similarly for \(p\otimes Id\). We have
\[(Id\otimes p)^*\circ\beta_{j+1}(u\times u)=\beta_{j+1}\circ(Id\otimes p)^*(u\times u)=\beta_{j+1}(u)=\rho^jv,\]
so does \(p\otimes Id\). This shows that
\[\beta_{j+1}(u\times u)=\rho^j(u\times v+v\times u).\]

If \(j\leq k,i>j+1\), we have by induction
\[E_{i}^{p,q}(\mathbb{Z}/2^j\eta\otimes\mathbb{Z}/2^k\eta)=\begin{cases}\mathbb{Z}/2\cdot\rho^{q-1}(u\times v\oplus v\times u)&p=q+1,q\leq j\\\mathbb{Z}/2\cdot\rho^{q-1}\frac{u\times v\oplus v\times u}{u\times v+v\times u}&p=q+1,q>j\\\mathbb{Z}/2\cdot\rho^{q-2}v\times v&p=q+2,q\leq j+1\\0&\textrm{else}\end{cases}\]
\[\beta_i(u\times v)=\beta_i\circ(Id\times\partial)^*(u)=(Id\times\partial)^*\circ\beta_i(u)=0\]
\[\beta_i(v\times u)=\beta_i\circ(\partial\times Id)^*(u)=(\partial\times Id)^*\circ\beta_i(u)=0,\]
which concludes the proof.
\end{proof}
\begin{proposition}\label{leibniz}
Let \(A, B\) be Tate MW-motives and \(x\in E(\textrm{W})^{*,*}_i(A), y\in E(\textrm{W})^{*,*}_i(B)\). The exterior product on \(E(\textrm{W})^{*,*}_i(A)\) is defined and we have
\[\beta_i(x\times y)=x\times\beta_i(y)+\beta_i(x)\times y.\]
\end{proposition}
\begin{proof}
By Proposition \ref{bounded}, there is a map
\[A\xrightarrow{f}\oplus_i\mathbb{Z}(i)[2i]\oplus\oplus_{i,j}\mathbb{Z}/2^j\eta(i)[2i]\]
such that \(H^*(f,\textbf{W})\) and \(E_i^{*,*}(f),i\geq 2\) is an isomorphism. So does \(f\otimes B\) for any Tate \(B\). Then the statement follows by Proposition \ref{leibniz1}.
\end{proof}
\begin{proposition}\label{bockstein compare}
Suppose that \(A,B\) are Tate MW-motives.
\begin{enumerate}
\item There is a Cartesian square
\[
	\xymatrix
	{
		\mathrm{H}^{*,*}_{\mathrm{MW}}(A,\mathbb{Z})\ar[r]\ar[d]	&\mathrm{H}^{*,*}_{\mathrm{M}}(A,\mathbb{Z})\ar[d]\\
		\mathrm{H}^{*,*}_{\mathrm{W}}(A,\mathbb{Z})\ar[r]				&\mathrm{H}^{*,*}_{\mathrm{M}}(A,\mathbb{Z}/2).
	}
\]
\item Let \(\rho=[-1]\in\mathbf{K}_1^{\mathrm{MW}}(k), i\geq 2\). For every \(a\in\mathbb{Z}\), \(\oplus_qE_i^{q+a,q}(A)\) is a graded \(\mathbb{Z}/2[\rho]\)-module. So,  \(E_i(A)^a:=\oplus E^{*+a,*}_i(A)\) can be regarded as a \(\mathbb{Z}/2[\rho]\)-module chain complex. We have
\[E_2(A)\otimes_{\mathbb{Z}/2[\rho]}E_2(B)=E_2(A\otimes B)\]
as complexes and
\[E_i(A)\otimes^L_{\mathbb{Z}/2[\rho]}E_i(B)=E_i(A\otimes B)\]
in \(D(\mathbb{Z}/2[\rho]\textrm{-mod})\) for every \(i\geq 2\).
\end{enumerate}
\end{proposition}
\begin{proof}
\begin{enumerate}
\item We have the distinguished triangle
\[\mathrm{H}\widetilde{\mathbb{Z}}\xrightarrow{\alpha,\gamma}\H_{\mathrm{W}}\mathbb{Z}\oplus \H_{\mu}\mathbb{Z}\longrightarrow \H_{\mu}\mathbb{Z}/2\xrightarrow{\partial[1]}\]
where the \(\partial\) is given by the composite
\[\H_{\mu}\mathbb{Z}/2[-1]\xrightarrow{\beta} \H_{\mu}\mathbb{Z}\xrightarrow{h}\H\widetilde{\mathbb{Z}}.\]
Suppose that \(x\in \H^{*,*}_{\mathrm{M}}(A,\mathbb{Z}/2)\). By \(\gamma(\partial(x))=0\) and Proposition \ref{bockstein}, we have
\[\partial(x)\in \mathrm{Im}(\eta)\cap \ker(\eta^i)\]
for some \(i\) large enough. It is in \(\mathrm{Im}(\eta^2)\cap \ker(\eta^i)\) if and only if its class in \(E_2^{*,*}(A)\) vanishes. The latter follows from \textit{loc. cit.} and \(\alpha(\partial(x))=0\). Iterating this procedure we see that \(\partial(x)=0\), which concludes the proof.
\item We construct an abstract isomorphism. It suffices to consider the case \(A,B=\mathbb{Z},\mathbb{Z}/2^j\eta\) by Proposition \ref{bounded}. Denote by \(S_j\) the complex
\[0\longrightarrow\mathbb{Z}/2[\rho]\xrightarrow{\rho^j}\mathbb{Z}/2[\rho]\longrightarrow 0\]
and by \(S\) the complex
\[0\longrightarrow\mathbb{Z}/2[\rho]\xrightarrow{0}\mathbb{Z}/2[\rho]\longrightarrow 0,\]
both concentrated on degree \(0,1\). We have equalities
\[S\otimes S=S\oplus S[1], S\otimes S_j=S_j\oplus S_j[1], S_j\otimes S_j=S_j\otimes S\]
where the last one is given by the quasi-isomorphism
\[
	\xymatrix
	{
		\mathbb{Z}/2[\rho]\ar[r]^-{(\rho^j,\rho^j)}\ar@{=}[d]	&\mathbb{Z}/2[\rho]\oplus\mathbb{Z}/2[\rho]\ar[r]^-{\rho^j+\rho^j}\ar[d]_{f}	&\mathbb{Z}/2[\rho]\ar@{=}[d]\\
		\mathbb{Z}/2[\rho]\ar[r]^-{(\rho^j,0)}							&\mathbb{Z}/2[\rho]\oplus\mathbb{Z}/2[\rho]\ar[r]^-{0+\rho^j}							&\mathbb{Z}/2[\rho]
	}
\]
where \(f(x,y)=(x,x+y)\). For every \(i\geq 2,j>0\), we have
\[E_i(\mathbb{Z})=\mathbb{Z}/2[\rho]\]
\[E_i(\mathbb{Z}/2^j\eta)=\begin{cases}S&\textrm{if }i<j+1\\S_j&\textrm{if }i\geq j+1\end{cases}\]
in \(D(\mathbb{Z}/2[\rho]\textrm{-mod})\) by Proposition \ref{higher} and \cite[Proposition 2.13]{Y2}. Suppose \(A=\mathbb{Z}/2^a\eta\), \(B=\mathbb{Z}/2^b\eta\) and \(0<a\leq b\). Then
\[A\otimes B=\mathbb{Z}/2^a\eta(1)[2]\oplus\mathbb{Z}/2^a\eta\]
by Proposition \ref{wittpart}. So
\[E_i(A\otimes B)=\begin{cases}S\oplus S[1]&i<a+1.\\S_a\oplus S_a[1]&i\geq a+1.\end{cases}\]
We have
\[LHS=\begin{cases}S\otimes S&i<a+1\\S_a\otimes S&a+1\leq i<b+1\\S_a\otimes S_a&i\geq b+1\end{cases}=RHS.\]
So we conclude the proof.
\end{enumerate}
\end{proof}
\begin{lemma}\label{truncated}
Suppose that \(A\) is Tate and that \(2\leq i\leq j+1\). We have an exact sequence
\[0\longrightarrow E^{*-2,*-1}_i(A)\longrightarrow E^{*,*}_i(A/2^j\eta)\longrightarrow E^{*,*}_i(A)\longrightarrow 0.\]
If we denote by \(u\in E^{0,0}_i(\mathbb{Z}/2^j\eta), v\in E^{2,1}_i(\mathbb{Z}/2^j\eta)\) the nonzero elements, we obtain
\[
E^{*,*}_i(A/2^j\eta)=E^{*,*}_i(A)\times u\oplus E^{*-2,*-1}_i(A)\times v
\]
and there is a long exact sequence
\[\ldots\xrightarrow{\rho^j}E_{j+2}^{*-2,*-1}(A)\longrightarrow E_{j+2}^{*,*}(A/2^j\eta)\longrightarrow E_{j+2}^{*,*}(A)\xrightarrow{\rho^j}E_{j+2}^{*+j,*+j}(A).\]
\end{lemma}
\begin{proof}
The second statement directly follows from the first one. For the first statement, which we prove by induction on \(j\), we have a diagram between distinguished triangles
\[
	\xymatrix
	{
		A\ar[r]\ar@{=}[d]	& A/2^j\eta\ar[r]\ar[d]	&A(1)[2]\ar[r]\ar[d]_-2	&A[1]\ar@{=}[d] \\
		A\ar[r]\ar[r]			&A/2^{j-1}\eta\ar[r]		&A(1)[2]\ar[r]				& A[1]
	}
\]
 If \(i=2\), we have a commutative diagram with exact rows
\[
	\xymatrix
	{
		E^{*,*}_2(A)\ar[r]				&E^{*+2,*+1}_2(A/2^j\eta)\ar[r]	&E^{*+2,*+1}_2(A)\ar[r]							&E^{*+2,*+1}_2(A)\\
		E^{*,*}_2(A)\ar[r]\ar[u]^0	&0\ar[r]\ar[u]						&E^{*+2,*+1}_2(A)\ar@{=}[r]\ar@{=}[u]	&E^{*+2,*+1}_2(A),\ar[u]^0
	}
\]
so the statement holds. For general \(j\), we prove the result by induction on \(i\). We have a commutative diagram with exact rows
\[
	\xymatrix
	{
		E^{*,*}_i(A)\ar[r]				&E^{*+2,*+1}_i(A/2^j\eta)\ar[r]					&E^{*+2,*+1}_i(A)\ar[r]					&E^{*+i,*+i-1}_i(A)\\
		E^{*,*}_i(A)\ar[u]^0\ar[r]	&E^{*+2,*+1}_i(A/2^{j-1}\eta)\ar[r]\ar[u]	&E^{*+2,*+1}_i(A)\ar[r]\ar@{=}[u]	&E^{*+i,*+i-1}_i(A)\ar[u]^0
	}
\]
by taking cohomologies on the statement for \(i-1\), which gives the statement for \(i\).

For the last statement, suppose \(2\leq i\leq j+1\). Let us consider the map
\[E_i^{*,*}(A)\oplus E_i^{*+2,*+1}(A)\xrightarrow{\beta^i}E_i^{*+i+1,*+i}(A)\oplus E_i^{*+i+3,*+i+1}(A)\]
by the decomposition above. We have
\begin{align*}&\beta^i(x\times v+y\times u)\\=&\beta^i(x)\times v+\beta^i(y)\times u+x\times\beta^i(v)+y\times\beta^i(u)\\=&\begin{cases}\beta^i(x)\times v+\beta^i(y)\times u&i<j+1\\\beta^i(x)\times v+\beta^i(y)\times u+\rho^j\cdot y\times v&i=j+1\end{cases}\end{align*}
by Proposition \ref{higher} and Proposition \ref{leibniz}. Then the statement follows from applying Snake lemma to the exact sequence before on \(j+1^{th}\)-page.
\end{proof}
Denote by \(Z_j(A), B_j(A)\subseteq E_1(A)\) the cycles and boundaries of \(j\)-th page of the Bockstein spectral sequence of \(A\) and by \(\beta^{j+1}:Z_j(A)/B_j(A)\longrightarrow Z_j(A)/B_j(A)\) the higher Bockstein. If \(A\) is Tate, \(Z_1(A)^{2n,n}=\ker(Sq^2\circ\pi)(A)^{2n,n}\).
\begin{proposition}\label{highermw}
Suppose that \(A\) is Tate and that \(j>0\). There is a surjection
\[[A,\mathbb{Z}/2^j\eta(n)[2n]]_{\widetilde{\mathrm{DM}}(k,\mathbb{Z})}\longrightarrow V(A,j)^n\times_{E^{2n+2,n+1}_{j+2}(A/2^j\eta)}\mathrm{H}^n(A,\mathbf{W}/2^j)\]
where
\[V(A,j)^n=\{(x,y)\in Z_j(A)^{2n,n}\oplus Z_{j+1}(A)^{2n+2,n+1}|\beta^{j+1}(x)=\rho^j\cdot y\}.\]
\end{proposition}
\begin{proof}
First, \(\mathbb{Z}/2^j\eta\) has a strong dual \(\mathbb{Z}/2^j\eta(-1)[-2]\) by \cite[Proposition 1.2]{L1}. So we have an isomorphism
\[[A,\mathbb{Z}/2^j\eta(n)[2n]]_{\widetilde{\mathrm{DM}}(k,\mathbb{Z})}=\widetilde{\mathrm{CH}}^{n+1}(A/2^j\eta).\]
By Corollary \ref{identification}, it suffices to prove that the kernel \(K\) of the Bockstein
\[E_1^{2n+2,n+1}(A/2^j\eta)\longrightarrow\widetilde{\mathrm{CH}}^{n+2}(A/2^j\eta)\]
is equal to \(V(A,j)^n\).

We know that by \textit{loc. cit.}
\[K=Z_{\infty}(A/2^j\eta)^{2n+2,n+1}=Z_{j+1}(A/2^j\eta)^{2n+2,n+1}.\]
We have a K\"unneth decomposition
\[\mathrm{CH}^{n+1}(A/2^j\eta)=\mathrm{CH}^n(A)\oplus \mathrm{CH}^{n+1}(A).\]
By the proof in Proposition \ref{higher}, \(Sq^2\) acts trivially on \(\mathbb{Z}/2^j\eta\) if \(j>0\). So we have
\[Z_1(A/2^j\eta)^{2n+2,n+1}=Z_1(A)^{2n,n}\oplus Z_1(A)^{2n+2,n+1}.\]

Suppose that we have \(K\subseteq Z_i(A)^{2n,n}\oplus Z_i(A)^{2n+2,n+1}\) for some \(1\leq i<j\) and \((x,y)\in K\). Then by the computation in Lemma \ref{truncated},
\[\beta^{i+1}(x,y)=(\beta^{i+1}(x),\beta^{i+1}(y))=0,\]
so \(K\subseteq Z_{j}(A)^{2n,n}\oplus Z_{j}(A)^{2n+2,n+1}\) by induction. On the other hand, the equation
\[\beta^{j+1}(x,y)=(\beta^{j+1}(x)+\rho^j\cdot y,\beta^{j+1}(y))=0\]
concludes the proof.
\end{proof}
\begin{theorem}\label{tate}
Every Tate MW-motive \(A\) is a direct sum of \(\mathbb{Z}(i)[2i]\), \(\mathbb{Z}/2^r\eta(i)[2i]\) and \(\mathbb{Z}/\mathbf{p}^r[i]\) where \(p\) is an odd prime.
\end{theorem}
\begin{proof}
Suppose \(A\) has up to \(\eta^r\)-torsion. We can find a map
\[f:\oplus_{p,s,i}\mathbb{Z}/\mathbf{p}^s[i]\longrightarrow A\]
where \(p\) is an odd prime and \(\mathrm{H}^{*}(C(f),\mathbf{W})\) has no odd torsion. Then \(f\) is injective by Proposition \ref{wittpart}. We replace \(A\) by \(C(f)\).

By the proof of Proposition \ref{bockstein}, we have
\[\frac{Ker(Sq^2\circ\pi)^{2*,*}(A)}{\pi^{-1}(\mathrm{Im}(Sq^2))^{2*,*}(A)}=E^{2*,*}(A).\]
We have a commutative diagram for \(A\)
\[
	\xymatrix
	{
		\pi^{-1}(\mathrm{Im}(Sq^2))^{2n,n}\ar[r]^-{\varphi}\ar[dr]_{\pi}	&\mathrm{CH}^{n-1}/\ker(Sq^2\circ\pi)^{2n-2,n-1}\ar[d]_{Sq^2\circ\pi}^{\cong}\\
																						&\mathrm{Im}(Sq^2)^{2n,n}.
	}
\]
Now \(\mathrm{CH}^*/\ker(Sq^2\circ\pi)^{2*,*}=\mathrm{Im}(Sq^2)^{2*,*}\subseteq \mathrm{Ch}^*\) is \(2\)-torsion so there is a basis \(\{y_{n,i},z_{n,j},x_{n,k}\}\) of \(\mathrm{CH}^n(A)\) such that
\[\sum_i\alpha_iy_{n,i}+\sum_j\beta_jz_{n,j}+\sum_k\lambda_kx_{n,k}\in \ker(Sq^2\circ\pi)^{2n,n}\]
if and only if \(\alpha_i\) are even for every \(i\), that \(\{z_{n,j}\}\subseteq\pi^{-1}(\mathrm{Im}(Sq^2))^{2n,n}\) and that the image of \(\{x_{n,k}\}\) freely generates \(E^{2n,n}_2(A)\). Then \(\pi^{-1}(\mathrm{Im}(Sq^2))^{2n,n}/(\oplus_i2\mathbb{Z}y_{n,i}+\oplus_k2\mathbb{Z}x_{n,k})=\oplus_j\mathbb{Z}z_{n,j}\) is a free \(\mathbb{Z}\)-module, the rank of which is the same that of \(\mathrm{CH}^{n-1}/\ker(Sq^2\circ\pi)^{2n-2,n-1}\) as a \(\mathbb{Z}/2\)-module. So there is a basis \(\{w_{n,j}\}\) of \(\pi^{-1}(\mathrm{Im}(Sq^2))^{2n,n}/(\oplus_i2\mathbb{Z}y_{n,i}+\oplus_k2\mathbb{Z}x_{n,k})\) satisfying
\[\varphi(w_{n,j})=\pi(y_{n-1,j}).\]
Then we obtain classes
\[t_{n,i}=(y_{n,i},w_{n+1,i})\in E_1^{2n,n}(A)\]
such that
\[\ker(Sq^2\circ\pi)^{2n,n}(A)=(\oplus_k\mathbb{Z}x_{n,k})\oplus(\oplus_i2\mathbb{Z}y_{n,i})\oplus(\oplus_j\mathbb{Z}w_{n,j}).\]
Hence
\[\mathrm{CH}^n(A)=\left(\oplus_k\mathbb{Z}x_{n,k}\right)\oplus\left(\oplus_i\mathbb{Z}y_{n,i}\right)\oplus\left(\oplus_j\mathbb{Z}w_{n,j}\right).\]

In general, suppose we have \(\mathbb{Z}\)-linearly independent elements \(\{x_{n,k}\}\subseteq Z_{t-1}(A)^{2n,n},t\geq 2\), such that for fixed \(n\), their classes in \(E_t^{2n+*,n+*}(A)\) are \(\mathbb{Z}/2\)-linearly independent and generate the \(\mathbb{Z}/2[\rho]\)-free part of \(E_t^{2n+*,n+*}(A)\). Define \(K\) by the exact sequence
\[0\longrightarrow K\longrightarrow\oplus_k\mathbb{Z}\cdot x_{n,k}\xrightarrow{\beta^t}E_t^{2n+t+1,n+t}(A).\]
By Lemma \ref{linear}, there is a decomposition
\[\oplus_k\mathbb{Z}\cdot x_{n,k}=U_n\oplus V_n\]
such that
\[\begin{array}{cc}K=2U_n\oplus V_n,&U_n/2=Im(\beta^t)^{2n+t+1,n+t}\end{array}.\]
Since \(2(n-t)<2n-t-1,t\geq 2\), we have \(E_t^{2n-t-1,n-t}(A)=0\), thus
\[\ker(E_t^{2n,n}(A)\xrightarrow{\beta^t}E_t^{2n+t+1,n+t}(A))=E_{t+1}^{2n,n}(A).\]
By Corollary \ref{dim}, there are a basis \(\{u_{n,i}\}\) of \(U_n\) and \(\mathbb{Z}\)-linearly independent elements \(\{v_{n+1,i}\}\subseteq V_{n+1}\) whose images in \(E_{t+1}^{2n+2+*,n+1+*}(A)\) span its \(\mathbb{Z}/2[\rho]/\rho^{t-1}\) components and satisfy
\[\rho^{t-1}\cdot v_{n+1,i}=\beta^t(u_{n,i})\]
in \(E_t^{*,*}(A)\). By Lemma \ref{linear}, there is a decomposition
\[V_{n+1}=(\oplus_i\mathbb{Z}\cdot v_{n+1,i})\oplus W_{n+1}\]
such that the image of \(W_{n+1}\) in \(E_{t+1}^{2n+2+*,n+1+*}\) freely spans its \(\mathbb{Z}/2[\rho]\)-free part.

We can find \(f_{t,n,i}\in[A,\mathbb{Z}/2^{t-1}\eta(n)[2n]]_{MW}\), corresponding to \((u_{n,i},v_{n+1,i})\in V(A,t-1)^n\) and \(c_i\in \H^n(A,\textbf{W}/2^{t-1})\) as in Proposition \ref{highermw}, such that \(\beta(c_i)\) freely generate the \(\mathbb{Z}/2^{t-1}\) components of \(\H^{n+1}(A,\textbf{W})\). Let us explain this in detail. By Corollary \ref{dim}, we see that
\[\#\{\mathbb{Z}/2^{t-1}\textrm{-components in }\H^{n+1}(A,\textbf{W})\}=dim(U_n).\]
The generators in LHS come from Bockstein images of elements \(c\in \H^{n}(A,\textbf{W}/2^{t-1})\), such that there is no \(c'\in \H^n(A,\textbf{W}/2^{t'-1}),t'<t\) such that \(2^{t-t'}c'\) and \(c\) differs by an integral class. We have commutative diagrams
\[
	\xymatrix
	{
		\H^n(A,\textbf{W}/2^{t'-1})\ar[r]^{2^{t-t'}}\ar[d]	&\H^n(A,\textbf{W}/2^{t-1})\ar[d]\ar[r]	&\H^{n+1}(A,\textbf{W})\ar[d]\\
		E^{2n+2,n+1}_{t+1}(A/2^{t'-1}\eta)\ar[r]^0				&E^{2n+2,n+1}_{t+1}(A/2^{t-1}\eta)\ar[r]&E^{2n+2,n+1}_{t+1}(A)
	}
\]
since \(E^{2n+2,n+1}_{\infty}(A)\subseteq E^{2n+2,n+1}_{t+1}(A)\) by the disscussion above. So we see that \(c_i\) satisfies the conditions of \(c\) by the linear independence of \(\{v_{n+1,i}\}\).

We iterate the procedure by setting \(\{x_{n,k}\}\) to be a basis of \(W_n\) and \(t\) increasing by one, until \(t=r+1\). The remaining subgroup in \(\mathrm{CH}^n(A)\) has the same dimension as the \(\mathbb{Z}/2[\rho]\)-free part of \(E^{2n,n}_{\infty}(A)\), which is the rank of the free part of \(\H^n(A,\textbf{W})\) by Corollary \ref{dim}. So we obtain \(g_{n,j}\in\widetilde{\mathrm{CH}}^n(A)\) as above by Corollary \ref{identification}.

Finally, the map
\[A\xrightarrow{t_{n,i},f_{t,n,i},g_{n,j}}(\oplus_{n,i}\mathbb{Z}/\eta(n)[2n])\oplus(\oplus_{t,i,n}\mathbb{Z}/2^{t-1}\eta(n)[2n])\oplus(\oplus_{n,j}\mathbb{Z}(n)[2n])\]
is an isomorphism by Proposition \ref{faithful}.
\end{proof}
\begin{corollary}\label{degeneracy}
For any Tate MW-motive \(A\) and \(r\in\mathbb{N}\), the following statements are equivalent:
\begin{enumerate}
\item The ring \(\H^*(A,\mathbf{W})\) has up to \(2^r\)-torsion;
\item The MW-motive \(A\) has up to \(\eta^{r+1}\)-torsion;
\item The Bockstein spectral sequence degenerates at the \((r+2)\)nd-page.
\end{enumerate}
\end{corollary}
\begin{remark}
There exists a projective smooth scheme whose MW-motive is Tate and Witt cohomology has odd torsion. For example, let \(E=S^3(U)\) where \(U\) is the tautological bundle of \(Gr(2,4)\). We have
\[e(E)=3e(U)^2\]
by \cite[Theorem 8.1]{L}. By \cite[Theorem 5.14]{Y1}, we easily obtain that
\[\H^4(\mathbb{P}(E),\mathbf{W})=\mathbf{W}(k)/3=\mathbb{Z}/3.\]
\end{remark}
\section{Vector Bundles on \(\mathbb{HP}^1\)}
If \(X\) is a smooth affine scheme, we denote by \(\mathcal{V}_n(X)\) the set of isomorphism classes of rank \(n\) bundles on \(X\). Recall from \cite{AHW} that there is a bijection
\[[X_+,\mathrm{BGL}_n]_{\mathbb{A}^1}\cong\mathcal{V}_n(X)\]
where the left-hand side denotes the set of \textit{pointed} classes of maps in the pointed motivic homotopy category \(\mathcal{H}_{\bullet}(k)\). Instead of $\mathrm{BGL}_n$, one may consider $\mathrm{BSL}_n$ and obtain
\[
[X_+,\mathrm{BSL}_n]_{\mathbb{A}^1}\cong\mathcal{V}_n^o(X)
\]
where \(\mathcal{V}_n^o(X)\) denotes the set of isomorphism classes of oriented bundles (i.e. with trivial determinant) of rank \(n\) on \(X\) (e.g. \cite[Theorem 4.2]{AF1}).
Now, we have a fiber sequence
\[\mathrm{BSL}_n\longrightarrow \mathrm{BGL}_n\longrightarrow \mathrm{B}\mathbb{G}_m\]
which can be used to compute the first homotopy sheaves of \(\mathrm{BGL}_n\). Following \cite[Discussion before Proposition 6.3]{AF2}, we see that the map \(\mathrm{BGL}_n\longrightarrow \mathrm{B}\mathbb{G}_m\) induces an isomorphism on \(\pi_1^{\mathbb{A}^1}\) and that \(\mathrm{BSL}_n\) is simply connected. More precisely, we obtain the following result.
\begin{lemma}
We have
\[\begin{array}{cc}\pi_i^{\mathbb{A}^1}(\mathrm{BSL}_2)=\begin{cases}*&i=0,\\1&i=1,\\\mathbf{K}_2^{\mathrm{MW}}&i=2,\end{cases}&\pi_i^{\mathbb{A}^1}(\mathrm{BSL}_n)=\begin{cases}*&i=0,\\1&i=1,\\\mathbf{K}_2^{\mathrm{M}}&i=2,\end{cases} \text{ for $n\geq 3$.}\end{array}\]
\end{lemma}
\begin{proof}
For the \(n\geq 3\) part, we apply the fiber sequence
\[\mathrm{SL}_n\longrightarrow \mathrm{SL}_{n+1}\longrightarrow\mathbb{A}^{n+1}\setminus 0\]
and \cite[Corollary 2.4]{AF2}.
\end{proof}

One of the most convenient way to compute morphisms in the pointed homotopy category is the so-called Postnikov tower associated to a pointed space. For \(\mathrm{BSL}_2\), it consists of pointed spaces \(\mathrm{BSL}_2^{(i)}\) for any \(i\geq 0\) and the only relevant points for us are the following:
\begin{enumerate}
\item\(\mathrm{BSL}_2^{(i)}=*\) for \(i=0,1\).
\item\(\mathrm{BSL}_2^{(2)}=K(\mathbf{K}_2^{\mathrm{MW}},2)\), where the latter is the Eilenberg-Maclane spaces associated to the (strictly \(\mathbb{A}^1\)-invariant) sheaf \(\mathbf{K}_2^{\mathrm{MW}}\). It has the property that \([X_+,\mathrm{K}(\mathbf{K}_2^{\mathrm{MW}},2)]_{\mathcal{H}_{\bullet}(k)}\cong \H^2(X,\mathbf{K}_2^{\mathrm{MW}})=\widetilde{\mathrm{CH}}^2(X)\).
\item There are principal fiber sequences
\[\mathrm{BSL}_2^{(i+1)}\longrightarrow \mathrm{BSL}_2^{(i)}\xrightarrow{\kappa_i}K(\pi_{i+1}^{\mathbb{A}^1}(\mathrm{BSL}_2),i+2)\]
where \(\kappa_i\) is the so-called \(k\)-invariant.
\item There are morphisms \(p_i:\mathrm{BSL}_2\longrightarrow \mathrm{BSL}_2^{(i)}\) for any \(i\in\mathbb{N}\) such that the following diagram commutes
\[
	\xymatrix
	{
		\mathrm{BSL}_2\ar@{=}[d]\ar[r]^-{p_{i+1}}	&\mathrm{BSL}_2^{(i+1)}\ar[d]\\
		\mathrm{BSL}_2\ar[r]_-{p_{i}}				&\mathrm{BSL}_2^{(i)}
	}
\]
and \(\mathrm{BSL}_2\cong \mathrm{holim}_i\mathrm{BSL}_2^{(i)}\).
\item If \(n\geq 3\), the same results hold for \(\mathrm{BSL}_n\) except that \(\mathrm{BSL}_n^{(2)}=\mathrm{K}(\mathbf{K}_2^{M},2)\).
\end{enumerate}
\begin{theorem}\label{rk2}
For \(X=\mathbb{HP}^1\), the Euler class induces a bijection
\[e:\mathcal{V}_2^o(X)\cong\widetilde{\mathrm{CH}}^2(X)\cong\mathbf{GW}(k)\]
and the second Chern class induces a bijection
\[c_2:\mathcal{V}_n^o(X)\cong \mathrm{CH}^2(X)\cong\mathbb{Z},n\geq 3.\]
\end{theorem}
\begin{proof}
We use the Postnikov tower as above, evaluating it at \(X_+\). In view of the first property, we find \([X_+,\mathrm{BSL}_2^{(i)}]_{\mathcal{H}_{\bullet}(k)}=*\) for \(i=0,1\), while
\[
[X_+,\mathrm{BSL}_2^{(2)}]_{\mathcal{H}_{\bullet}(k)}=[X_+,K(\mathbf{K}_2^{MW},2)]_{\mathcal{H}_{\bullet}(k)}\cong \H^2(X,\mathbf{K}_2^{MW})=\widetilde{\mathrm{CH}}^2(X).
\]
We may now use the fiber sequences
\[
\mathrm{BSL}_2^{(i+1)}\longrightarrow \mathrm{BSL}_2^{(i)}\xrightarrow{\kappa_i}K(\pi_{i+1}^{\mathbb{A}^1}(\mathrm{BSL}_2),i+2)
\]
for any \(i\geq 2\) to obtain an exact sequence
\[\xymatrix
{[X_+,K(\pi_{i+1}^{\mathbb{A}^1}(\mathrm{BSL}_2),i+1)]_{\mathcal{H}_{\bullet}(k)}\ar[r]&[X_+,\mathrm{BSL}_2^{(i+1)}]_{\mathcal{H}_{\bullet}(k)}\ar[r]^{\varphi}&[X_+,\mathrm{BSL}_2^{(i)}]_{\mathcal{H}_{\bullet}(k)}\ar[d]\\&&[X_+,K(\pi_{i+1}^{\mathbb{A}^1}(\mathrm{BSL}_2),i+2)]_{\mathcal{H}_{\bullet}(k)}}.
\]
Now, we know that \(X\cong(\mathbb{P}^1)^{\wedge2}\) by \cite{ADF17} and we obtain for any strictly \(\mathbb{A}^1\)-invariant sheaf \(\mathbf{F}\) and any \(j\in\mathbb{N}\)
\[
[X_+,K(\mathbf{F},j)]_{\mathcal{H}_{\bullet}(k)}\cong[(\mathbb{P}^1)^{\wedge2}_+,K(\mathbf{F},j)]_{\mathcal{H}_{\bullet}(k)}=\begin{cases}0&j\neq 0,2,\\ \mathbf{F}(k) & j=0, \\\mathbf{F}_{-2}(k)&j=2,\end{cases}
\]
in view of \cite[Lemma 4.5]{AF1}. It follows easily that \(\varphi\) is a bijection for any \(i\geq 2\) and that \(\H^2(X,\mathbf{K}_2^{\mathrm{MW}})=(\mathbf{K}_2^{\mathrm{MW}})_{-2}(k)=\mathbf{GW}(k)\). Thus
\[
[X_+,\mathrm{BSL}_2]_{\mathcal{H}_{\bullet}(k)}\cong[X_+,\mathrm{BSL}_2^{(2)}]_{\mathcal{H}_{\bullet}(k)}\cong\widetilde{\mathrm{CH}}^2(X)\cong\mathbf{GW}(k)
\]
and it remains to prove that the bijection is given by the Euler class. For this, we observe that \(\mathrm{SL}_2\simeq \mathbb{A}^2\setminus 0\) and we consequently obtain a fiber sequence \[\mathbb{A}^2\setminus0\longrightarrow*\longrightarrow \mathrm{BSL}_2\]
which allows to deduce that the map we consider is the first obstruction for an oriented bundle of rank \(2\) to be trivial. One may use \cite[Theorem 1]{AF3} to conclude. The statement for \(n\geq 3\) follows in the same fashion as above.
\end{proof}

\begin{corollary}
The Witt-valued Euler class induces a surjection
\[e:\mathcal{V}_2^o(\mathbb{HP}^1)\longrightarrow\mathbf{W}(k).\]
In particular, if \(k\) is Euclidean and \(n\in\mathbb{N}\), there is a rank \(2\) oriented vector bundle \(E_n\) such that
\(e(E_n)=2^n\in\mathbf{W}(k).\)
Hence we have
\[\H^2(\mathbb{P}(E_n),\mathbf{W})=\mathbb{Z}/2^n.\]
\end{corollary}
To conclude, let's observe that the identification of \(\mathcal{V}_n^o(\mathbb{HP}^1)\) allows to completely classify rank \(n\) bundles on \(\mathbb{HP}^1\). For this, observe that the fiber sequence
\[\mathrm{BSL}_n\longrightarrow \mathrm{BGL}_n\longrightarrow \mathrm{B}\mathbb{G}_m\]
yields an exact sequence
\[
[\mathbb{HP}^1_+,\mathbb{G}_m]_{\mathcal{H}_{\bullet}(k)}\to[\mathbb{HP}^1_+,\mathrm{BSL}_n]_{\mathcal{H}_{\bullet}(k)}\to[\mathbb{HP}^1_+,\mathrm{BGL}_n]_{\mathcal{H}_{\bullet}(k)}\to[\mathbb{HP}^1_+,\mathrm{B}\mathbb{G}_m]_{\mathcal{H}_{\bullet}(k)}.
\]
As \(\mathrm{B}\mathbb{G}_m\cong K(\mathscr{O}^{\times},1)\), we obtain \([\mathbb{HP}^1_+,\mathrm{B}\mathbb{G}_m]_{\mathcal{H}_{\bullet}(k)}=*\) and thus an exact sequence of pointed sets
\[[\mathbb{HP}^1_+,\mathbb{G}_m]_{\mathcal{H}_{\bullet}(k)}\to\mathcal{V}_n^o(\mathbb{HP}^1)\to\mathcal{V}_n^o(\mathbb{HP}^1)\to*.\]
\begin{lemma}
We have \([\mathbb{HP}^1_+,\mathbb{G}_m]_{\mathcal{H}_{\bullet}(k)}=\mathscr{O}(\mathbb{HP}^1)^{\times}=k^{\times}\) and the action of \(k^{\times}\) on \(\mathcal{V}_2^o(\mathbb{HP}^1)\cong\mathbf{GW}(k)\) is given by
\[\alpha\cdot a=\langle\alpha\rangle\cdot a\]
for any \(\alpha\in k^{\times}\) and \(a\in\mathbf{GW}(k)\). A similar statement holds for \(\mathcal{V}_3^o(\mathbb{HP}^1),n\geq 3\).
\end{lemma}
\begin{proof}
We know that \(\mathbb{G}_m\) is \(\mathbb{A}^1\)-rigid and thus that \([X_+,\mathbb{G}_m]_{\mathcal{H}_{\bullet}(k)}\cong\mathscr{O}(X)^{\times}\) for any \(X\in \mathrm{Sm}/k\). We know from \cite[Theorem 3.1]{PW} that  \(\mathscr{O}(\mathbb{HP}^1)^{\times}=k^{\times}\). For the second assertion, we may appeal to \cite[\S 6]{AF2}.
\end{proof}
\begin{corollary}\label{vb}
The Euler class induces a bijection
\[\mathcal{V}_2(\mathbb{HP}^1)\cong\mathbf{GW}(k)/k^{\times}\]
and the second Chern class induces a bijection
\[\mathcal{V}_n(\mathbb{HP}^1)\cong\mathbb{Z}, n\geq 3.\]
So \(E\in\mathcal{V}_2(\mathbb{HP}^1)\) satisfies that \(E\oplus O_{\mathbb{HP}^1}\) is free and \(E\) is not free if and only if
\[0\neq e(E)\in\mathbf{I}(k)\subseteq\widetilde{\mathrm{CH}}^2(\mathbb{HP}^1),\]
which is equivalent to \(\mathbb{P}(E)\) admitting a genuine \(\eta^2\)-torsion when \(k\) is Euclidean.
\end{corollary}
\section{Blow-ups with Even Codimensional Center}
We now continue the study of the MW-motive of blow-ups initiated in \cite{Y1}. In the sequel, we suppose that \(Z,X\in \mathrm{Sm}/k\) and that \(Z\subseteq X\) is closed of \emph{even} codimension \(n\).
\begin{definition}
We set
\[M_Z(X)_{\eta}:=C(\mathbb{Z}(X\setminus Z)\longrightarrow\mathbb{Z}(X)\longrightarrow\mathbb{Z}(X)/\eta)\]
in \(\widetilde{\mathrm{DM}}(k,\mathbb{Z})\).
\end{definition}
\begin{lemma}\label{aux}
We have
\[\footnotesize M_Z(X)_{\eta}=C\left(\mathbb{Z}(X)\longrightarrow \mathrm{Th}(\det(N_{Z/X}))(n-1)[2n-2]\xrightarrow{\eta}\mathrm{Th}(\det(N_{Z/X}))(n-2)[2n-3]\right)(1)[1].\]
\end{lemma}
\begin{proof}
This follows from the following commutative diagram in which the rows and columns are distinguished triangles
\[
	\xymatrix
	{
																&\mathbb{Z}(X\setminus Z)(-1)[-1]\ar@{=}[r]\ar[d]	&\mathbb{Z}(X\setminus Z)(-1)[-1]\ar[d]\\
		\mathbb{Z}(X)\ar[r]^{\eta}\ar@{=}[d]	&\mathbb{Z}(X)(-1)[-1]\ar[r]\ar[d]							&\mathbb{Z}(X)/\eta(-1)[-1]\ar[d]\\
		\mathbb{Z}(X)\ar[r]							&\mathrm{Th}(\det(N_{Z/X}))(n-2)[2n-3]\ar[r]						&M_Z(X)_{\eta}(-1)[-1].
	}
\]
\end{proof}
\begin{proposition}\label{blowup}
Suppose that $Z\subset X$ is of even codimension and is quasi-projective. We have
\[\mathbb{Z}(Bl_Z(X))=M_Z(X)_{\eta}(-1)[-2]\oplus\bigoplus_{i=1}^{\frac{n}{2}-1}\mathbb{Z}(Z)/\eta(2i-1)[4i-2].\]
\end{proposition}
\begin{proof}
We have closed immersions \(Bl_Z(X)\subseteq Bl_Z(X\times\mathbb{A}^1)\) and \(Z\times\mathbb{A}^1\subseteq Bl_Z(X\times\mathbb{A}^1)\) with empty intersection. There is an \(\mathbb{A}^1\)-bundle \(Bl_Z(X\times\mathbb{A}^1)\setminus(Z\times\mathbb{A}^1)\longrightarrow Bl_Z(X)\). By \cite[Theorem 5.20]{Y1}, we have a commutative diagram whose rows and columns are distinguished triangles
\[\footnotesize
	\xymatrix
	{
																&\mathbb{Z}(Z)/\eta(n-1)[2n-2]\ar[d]\ar@{=}[r]																	&\mathbb{Z}(Z)/\eta(n-1)[2n-2]\ar[d]_{\partial}\\
		\mathbb{Z}(Bl_Z(X))\ar[r]\ar@{=}[d]	&\mathbb{Z}(Bl_Z(X\times\mathbb{A}^1))\ar[r]\ar[d]														&\mathrm{Th}(\det(N_{Z/X}))(n-1)[2n-2]\ar[d]\\
		\mathbb{Z}(Bl_Z(X))\ar[r]					&\mathbb{Z}(X)\oplus\bigoplus_{i=1}^{\frac{n}{2}-1}\mathbb{Z}(Z)/\eta(2i-1)[4i-2]\ar[r]^-{(\varphi,\psi_i)}	&C.\\
	}
\]
By the proof of \textit{loc. cit.}, the map \(\partial\) is the composite
\[\footnotesize\mathbb{Z}(Z)/\eta(n-1)[2n-2]\longrightarrow\mathbb{Z}(\mathbb{P}(N_{Z/X}\oplus O_Z))\longrightarrow\mathbb{Z}(Bl_Z(X\times\mathbb{A}^1))\longrightarrow \mathrm{Th}(\det(N_{Z/X}))(n-1)[2n-2].\]
The composition of the last two arrows above is the Gysin map of the section \(Z\subseteq\mathbb{P}(N_{Z/X}\oplus O_Z)\) by the Cartesian diagram
\[
	\xymatrix
	{
		Z\ar[r]^0\ar[d]									&Z\times\mathbb{A}^1\ar[d]\\
		\mathbb{P}(N_{Z/X}\oplus O_Z)\ar[r]	&Bl_Z(X\times\mathbb{A}^1)
	}
\]
and the excess intersection formula (\cite[Theorem 3.2]{F}). This coincides with the composite
\[\footnotesize\mathbb{Z}(Z)/\eta(n-1)[2n-2]=\mathrm{Th}(\det(N_{Z/X}))/\eta(n-2)[2n-4]\longrightarrow \mathrm{Th}(\det(N_{Z/X}))(n-1)[2n-2]\]
by the proof in \cite[Theorem 5.14]{Y1}. So we see that
\[C=\mathrm{Th}(\det(N_{Z/X}))(n-2)[2n-3].\]
By (b) in \textit{loc. cit.}, we have \(\psi_i=0\). Now, \(\varphi\) is the composite
\[\mathbb{Z}(X)\longrightarrow \mathrm{Th}(\det(N_{Z/X}))(n-1)[2n-2]\xrightarrow{\eta}\mathrm{Th}(\det(N_{Z/X}))(n-2)[2n-3]\]
since the section \(X\subseteq Bl_Z(X\times\mathbb{A}^1)\) constantly equal to \(1\) induces an isomorphism
\[M_Z(X)=M_{Z\times\mathbb{A}^1}(Bl_Z(X\times\mathbb{A}^1))\]
by the proof of \cite[Theorem 2.23, \S 3]{MV}. We finally apply Lemma \ref{aux} to conclude.
\end{proof}
\begin{corollary}
If $Z\subset X$ is of even codimension $n$ and is quasi-projective, we have
\[\mathbb{Z}(Bl_Z(X))=\mathbb{Z}(X\setminus Z)\]
in \(\widetilde{\mathrm{DM}}(k,\mathbb{Z})[\eta^{-1}]\).
\end{corollary}

\begin{remark}
If \(n\) is odd, we have
\[\mathbb{Z}(Bl_Z(X))=\mathbb{Z}(X)\]
in \(\widetilde{\mathrm{DM}}(k,\mathbb{Z})[\eta^{-1}]\) by \cite[Theorem 5.14]{Y1}.
\end{remark}
{}


\begin{thebibliography}{}
\bibitem{ADF} A. Asok, B. Doran, J. Fasel, {\it Smooth models of motivic spheres and the clutching construction}, IMRN 6(1):1890-1925, 2016.
\bibitem{AF1} A. Asok, J. Fasel, {\it Algebraic vector bundles on spheres}, J. Topology, 7(3):894-926, 2014.
\bibitem{AF2} A. Asok, J. Fasel, {\it A cohomological classification of vector bundles on smooth affine threefolds}, Duke Math. J., 163(14):2561-2601, 2014.
\bibitem{AF4} A. Asok, J. Fasel, {\it Secondary characteristic classes and the Euler class}, Documenta Mathematica, 2015, Vol. Extra Volume, pp.7-29.
\bibitem{AF3} A. Asok, J. Fasel, {\it Comparing Euler classes}, Quart. J. Math., 67:603-635, 2016.
\bibitem{ADF17} A. Asok, B. Doran, J. Fasel, {\it Smooth models of motivic spheres and the clutching construction}, IMRN 6(1):1890-1925, 2017. 
\bibitem{AHW} A. Asok, M. Hoyois, M. Wendt, {\it Affine representability results in \(\mathbb{A}^1\)-homotopy theory I: vector bundles}, Duke Math. J. 166(10): 1923-1953 (2017).
\bibitem{B} T. Bachmann, {\it The generalized slices of Hermitian K-theory}, Journal of Topology \textbf{10}, no. 4, 1124–1144 (2017).
\bibitem{BCDFO} T. Bachmann, B. Calm\`es, F. D\'eglise, J. Fasel, P.\,A. \O stvaer, {\it Milnor–Witt motives}, arXiv:2004.06634 (2020).
\bibitem{Bro} W. Browder, {\it Torsion in H-spaces}, Annals of Mathematics, Jul., 1961, Second Series, Vol. 74, No. 1 (Jul., 1961), pp. 24-51.
\bibitem{F} J. Fasel, {\it The excess intersection formula for Grothendieck–Witt groups}, Manuscripta Mathematica. 130, 411-423 (2009).
\bibitem{HMX} T. Hudson, A. Martirosian, H. Xie, {\it Witt groups of Spinor varieties}, Proceedings of the London Mathematical Society, Volume 125, Issue 5,
November 2022, Pages 1152-1178.
\bibitem{HW} J. Hornbostel, M. Wendt, {\it Chow–Witt rings of classifying spaces for symplectic and special linear groups}, Journal of Topology, Volume 12, Issue 3 (2019).
\bibitem{L1} M. Levine, {\it Motivic Euler characteristics and Witt-valued characteristic classes}, Nagoya Mathematical Journal, Volume 236: Celebrating the 60th Birthday of Shuji Saito, December 2019 , pp. 251-310.
\bibitem{L} M. Levine, {\it Aspects of enumerative geometry with quadratic forms}, Documenta Mathematica 25 (2020) 2179-2239.
\bibitem{Mo} F. Morel, {\it \(\mathbb{A}^1\)-Algebraic Topology over a Field}, volume 2052 of Lecture Notes in Math. Springer, New York (2012).
\bibitem{MV} F. Morel, V. Voevodsky, {\it\(\mathbb{A}^1\)-homotopy theory of schemes}, Publications Math\'ematiques de l’I.H.\'E.S. no \textbf{90} (1999).
\bibitem{T} M. Tavakol, {\it The Chow ring of the moduli space of curves of genus zero}, Journal of Pure and Applied Algebra, Volume 221, Issue 4, April 2017, Pages 757-772.
\bibitem{MW} A. K. Matszangosz, M. Wendt, {\it 4-Torsion classes in the integral cohomology of oriented Grassmannians}, arXiv:2403.06897 (2024).
\bibitem{PW} I. Panin, C. Walter, {\it Quaternionic Grassmannians and Borel classes in algebraic geometry}, St. Petersburg Math. J.33(2022), no.1, 97–140.
\bibitem{V1} V. Voevodsky, {\it On motivic cohomology with \(\mathbb{Z}/l\)-coefficients}, Annals of Mathematics, 174 (2011), 401-438.
\bibitem{Wei94} C. Weibel, {\it An introduction to homological algebra}, Cambridge Studies in Advanced Mathematics 38, Cambridge University Press, xiv, 450 pages.
\bibitem{Y} N. Yang, {\it Quaternionic projective bundle theorem and Gysin triangle in MW-motivic cohomology},  Manuscripta Math. 164 (2021), no.~1, 39--65.
\bibitem{Y1} N. Yang, {\it Projective bundle theorem in MW-motivic cohomology}, Documenta Mathematica 26 (2021) 1045-1083.
\bibitem{Y2} N. Yang, {\it Split Milnor-Witt motives and its applications to fiber bundles}, Cambridge Journal of Mathematics, Volume 10, Number 4, 935–1004, 2022.
\end{thebibliography}
\end{document}